\documentclass[11pt,reqno]{amsart}

\usepackage[toc,page]{appendix}

\usepackage{xcolor}
\usepackage{placeins}
\usepackage{amsfonts,amsmath,amsthm}
\usepackage{amssymb,epsfig}
\usepackage{enumerate} 
\usepackage{fullpage}
\usepackage[utf8]{inputenc}
\usepackage{mathpazo}
\usepackage{mathtools}

\usepackage{eucal}
\usepackage[unicode=true]{hyperref}
\hypersetup{colorlinks = true}
\hypersetup{
     colorlinks,
     linkcolor={black!10!red},
     linkbordercolor = {black!100!red},
     citecolor={blue}
}
\usepackage{pdfsync}

\usepackage{geometry}
\usepackage{graphicx}
\usepackage{epsfig}
\usepackage{tikz}
\usetikzlibrary{matrix}
\usepackage{caption}
\usepackage{color} %color
\definecolor{vert}{rgb}{0,0.6,0}

\usepackage{comment}
\numberwithin{figure}{section}

\theoremstyle{plain}
\newtheorem{thm}{Theorem}[section]

\newtheorem{defn}{Definition}
\newtheorem{quest}{Question}

\newtheorem{lem}[thm]{Lemma}
\newtheorem{cor}[thm]{Corollary}
\newtheorem{prop}[thm]{Proposition}
\theoremstyle{remark}
\newtheorem{rem}{\bf{Remark}}
\numberwithin{equation}{section}

\newcommand{\R}{\mathbb{R}}

\usepackage{enumitem}
\usepackage{tikz}
\usetikzlibrary{shapes.geometric, arrows, calc}
\tikzstyle{startstop} = [rectangle, rounded corners, minimum width=4cm, minimum height=2cm,text centered, draw=black, fill=red!30]
\tikzstyle{arrow} = [thick,->,>=stealth]
\makeatletter
\def\namedlabel#1#2{\begingroup
    #2%
    \def\@currentlabel{#2}%
    \phantomsection\label{#1}\endgroup
}
\makeatother

\usepackage{import}
\usepackage{xifthen}
\usepackage{pdfpages}
\usepackage{transparent}

\definecolor{darkgreen}{rgb}{0,0.4,0}

\UseRawInputEncoding

\begin{document}
\title{The regularity with respect to domains of the additive eigenvalues of superquadratic Hamilton--Jacobi equation}

\thanks{ $^\dagger$ Corresponding author.\\
%The authors are supported in part by . 
Dohyun Kwon was supported by the 2023 Research Fund of the University of Seoul. The work of Son N. T. Tu began at University of Wisconsin-Madison, and is supported in part by NSF grants DMS-1664424 and DMS-1843320. Farid Bozorgnia was supported by the Portuguese National Science Foundation through FCT fellowships.
}

\begin{abstract} 

We study the additive eigenvalues on changing domains, along with the associated vanishing discount problems. We consider the convergence of the vanishing discount problem on changing domains for a general scaling type $\Omega_\lambda = (1+r(\lambda))\Omega$ with a continuous function $r$ and a positive constant $\lambda$. We characterize all solutions to the ergodic problem on $\Omega$ in terms of $r$. In addition, we demonstrate that the additive eigenvalue $\lambda\mapsto c_{\Omega_\lambda}$ on a rescaled domain $\Omega_\lambda = (1+\lambda)\Omega$ possesses one-sided derivatives everywhere. Additionally, the limiting solution can be parameterized by a real function, and we establish a connection between the regularity of this real function and the regularity of $\lambda \mapsto c_{\Omega_\lambda}$. We provide examples where higher regularity is achieved.

\date{May 4, 2024}
\end{abstract}
%%%%%%%%%%%%%%%%%%%%%%%%%%%%%%%%%%%%%%%%%%%%%%%%%%%%%%%%%%%

\author{Farid Bozorgnia}
\address[F. Bozorgnia]
{
CAMGSD, Department of Mathematics, Instituto Superior Técnico, Lisbon, Portugal}
\email{farid.bozorgnia@tecnico.ulisboa.pt}

\author{Dohyun Kwon}
\address[D. Kwon]
{
Department of Mathematics, University of Seoul, 163 Seoulsiripdaero, Dongdaemun-gu, Seoul, 02504, Republic of Korea}
\email{dh.dohyun.kwon@gmail.com}

\author{Son N. T. Tu$^\dagger$}
\address[S. N.T. Tu]
{
Department of Mathematics, 
Michigan State University, 
619 Red Cedar Road, 
East Lansing, MI 48824, USA}
\email{tuson@msu.edu}

%\date{\today}
\maketitle
\setcounter{tocdepth}{1}

% \tableofcontents
\section{Introduction}

The vanishing discount problem concerns the behavior of the family of solutions as the discount factor goes to $0$. Let $v_\lambda$ be the solution to the state-constraint problem with the discount factor $\lambda>0$,
\begin{equation}\label{eq:main}\tag{$\lambda, \Omega$}
\begin{cases}
     \lambda v_\lambda(x) + \left|Dv_\lambda(x)\right|^p -f(x) - \varepsilon \Delta v_\lambda(x) \leq 0 \qquad\text{in}\;\Omega, \\
     \lambda v_\lambda(x) + \left|Dv_\lambda(x)\right|^p -f(x) - \varepsilon \Delta v_\lambda(x) \geq 0 \qquad\text{on}\;\overline{\Omega},
\end{cases}
\end{equation}
where $\Omega$ is an open, bounded, and connected domain in $\mathbb{R}^n$ with $\mathrm{C}^2$ boundary, $\varepsilon>0$ is fixed, $p>2$ and $f\in \mathrm{C}^1(\overline{\Omega})$. The Hamiltonian is given by $H(x,\xi) = |\xi|^p - f(x)$ for  $(x,\xi) \in \overline{\Omega}\times \R^n$. As shown in \cite{lasry_nonlinear_1989}, we have
$\lambda v_\lambda \to -c(0)$ and $v_{\lambda} - v_\lambda(x_0) \to v$ as $\lambda \to 0^+$ for a fixed $x_0\in \Omega$ where $v$ solves the ergodic problem:
\begin{equation}\label{eq:S_0}\tag{$0, \Omega$}
\begin{cases}
    H(x,Dv(x)) - \varepsilon \Delta v(x) \leq c_\Omega \qquad\text{in}\;\Omega,\\
    H(x,Dv(x)) - \varepsilon \Delta v(x) \geq c_\Omega \qquad\text{on}\;\overline{\Omega}.
\end{cases}
\end{equation}The additive eigenvalue denoted by $c(0)$ is defined as
\begin{equation}\label{eq:c(0)}
    c_\Omega = \min \big\lbrace c\in \mathbb{R}: H(x,Du(x)) - \varepsilon \Delta u(x) \leq c\;\;\text{in}\;\Omega\;\text{has a solution} \big\rbrace
\end{equation}
and it is also the unique constant where \eqref{eq:S_0} can be solved. 
% We study the additive eigenvalues $c(\lambda)$ on the domain $\Omega_\lambda$, a variation of $\Omega$, together with the associated vanishing discount problems on changing domains $\Omega_\lambda\to \Omega$. We show that there is a one-to-one connection between the two problems that are not available when $\varepsilon=0$. 
Let $c_{\Omega_\lambda}$ be also the unique constant such that the cell problem on $\Omega_\lambda$ below can be solved:
\begin{equation}\label{eq:cell_lambda}\tag{$0, \Omega_\lambda$}
\begin{cases}
   H(x,Du (x))  - \varepsilon\Delta u(x) \leq c_{\Omega_\lambda} \qquad\text{in}\;\;\Omega_\lambda,\\
   H(x,Du(x))  - \varepsilon\Delta u(x) \geq c_{\Omega_\lambda} \qquad\text{on}\;\overline{\Omega}_\lambda.
\end{cases} 
\end{equation} 

The authors of \cite{barles_large_2010} present some first properties of the additive eigenvalues with respect to the underlying domains when $1< p\leq 2$. To be precise, let $c_\Omega$ be the additive eigenvalue of $H$ on $\Omega$. They show that $c_\Omega \leq c_{\Omega'}$ if $\Omega\subset\Omega'$ and, in general, one has continuity with respect to the domain (under some appropriate perturbations). The deeper properties of this map (especially in the second-order case and with $p>2$) seem to remain uninvestigated. 

In the linear case, the smoothness of $\lambda\mapsto c_{\Omega_\lambda}$, the principal eigenvalue for elliptic equations, is vastly studied (see Section \ref{sec:the quad case}) for a general perturbation $\Omega_\lambda = \left\lbrace x+\lambda \textbf{V}(x): x\in \Omega\right\rbrace$. When $p=2$ and $\mathbf{V} = \mathbf{Id}$, the problem we consider can be transformed into this linear problem, and we obtain the smoothness of this map via the so-called Hopf--Cole transform. We note that solutions to \eqref{eq:main} have different behaviors when $p>2$ and $1<p\leq 2$, which solutions are also called \emph{large solution} that explode on the boundary. 
The case $1<p\leq 2$ will be addressed in future work.

In the case of the first order Hamilton–Jacobi equations, using Mather measures from weak KAM theory, it is known that if $\Omega_\lambda = (1+\lambda)\Omega$ then $\lambda \mapsto c_{\Omega_\lambda}$ is one-sided differentiable everywhere (see \cite[Theorem 1.4]{tu_vanishing_2021}).  

\subsection{Assumptions and main results} Throughout the paper, we make the following assumptions.

\begin{description}[style=multiline, labelwidth=1cm, leftmargin=2.0cm]
    \item[\namedlabel{itm:A1}{$(\mathcal{A}_1)$}] The domain $\Omega$ is a bounded star-shaped (with respect to the origin) open subset of $\mathbb{R}^n$, and there exists some $\kappa > 0$ such that $B(0,\kappa)\subset\subset \Omega$ and
    \begin{equation}\label{eq:geometric-condition}
        \mathrm{dist}(x,\overline{\Omega}) \geq \kappa r \qquad\text{for all}\; x\in (1+r) \partial\Omega, \;\text{for all}\;r>0. 
    \end{equation}
    The Hamiltonian is $H(x,\xi) = |\xi|^p - f(x)$ where $p>2$, $f\in \mathrm{C}^1(\overline{U})$ and $U\subset \mathbb{R}^n$ is an open, bounded set with smooth boundary such that $\overline{\Omega}\subset U$. 
    % \item[\namedlabel{itm:A2}{$(\mathcal{A}_2)$}] The Hamiltonian is $H(x,\xi) = |\xi|^p - f(x)$ where $p>2$, $f\in \mathrm{C}^1(\overline{U})$ and $U\subset \mathbb{R}^n$ is an open, bounded set with smooth boundary such that $\overline{\Omega}\subset U$. 
\end{description}
Condition \eqref{eq:geometric-condition} is introduced in \cite{capuzzo-dolcetta_hamiltonjacobi_1990} to ensure the well-posedness of \eqref{eq:main}. For Theorem \ref{thm:general}, addressing vanishing discounts, we use the following changing domain setup.
\begin{description}[style=multiline, labelwidth=1cm, leftmargin=2.0cm]
    \item[\namedlabel{itm:A2}{$(\mathcal{A}_2)$}] We consider $\Omega_\lambda := (1+r(\lambda))\Omega$ for $\lambda > 0$ such that $\Omega_\lambda \subset U$, and $r:[0,\infty)\rightarrow \mathbb{R}$ is continuous with $r(0) = 0$ and %$ \lim_{\lambda\rightarrow 0^+}\frac{r(\lambda)}{\lambda} = \gamma  \in (-\infty,+\infty)$.
    \begin{equation}\label{eq:gam}
        \lim_{\lambda\rightarrow 0^+}\frac{r(\lambda)}{\lambda} = \gamma  \in (-\infty,+\infty). 
    \end{equation}
\end{description}
We denote by $L(x,v)$ the Legendre transform of $H(x,\xi)$:
\begin{equation}\label{eq:formula_L}
    L(x,v) = C_p|v|^q + f(x), \quad \text{where}\quad C_p = p^{-1/q}(p-1), \; p^{-1}+q^{-1}=1.
\end{equation}
We write $\nabla L(x,v) = \big(D_xL(x,v), D_vL(x,v)\big)$ for $(x,v)\in \overline{\Omega}\times\R^n$. For a measure $\mu$ on $\overline{\Omega}\times \R^n$, we define 
\begin{equation}\label{eq:integration_notation_main}
    \langle \mu,  \varphi\rangle_\Omega := \int_{\overline{\Omega}\times \R^n}  \varphi(x,v)\;d\mu(x,v), \qquad\text{for}\; \varphi\in \mathrm{C}(\overline{\Omega}\times \R^n)\cap L^1(\mu).
\end{equation}
Let $\mathcal{M}(\Omega)$ be the set of \emph{viscosity Mather measures} (Definition \ref{def:M_0}). We first show the convergence for the following vanishing discount problem,
\begin{equation}\label{eq:S_lambda}\tag{$ \lambda , \Omega_\lambda$}
\begin{cases}
    \lambda  u_\lambda(x) +  H(x,Du_\lambda(x))  - \varepsilon\Delta u_\lambda(x) \leq 0 \qquad\text{in}\;\;\Omega_\lambda,\\
    \lambda  u_\lambda(x) +  H(x,Du_\lambda(x))  - \varepsilon \Delta u_\lambda(x) \geq 0 \qquad\text{on}\;\overline{\Omega}_\lambda.
\end{cases} 
\end{equation}
Similar to \cite{tu_vanishing_2021}, if $r(\lambda)$ and $\gamma$ are given in \ref{itm:A2} then $u_\lambda +  \lambda ^{-1}c_\Omega$ is bounded and convergent. %Its limit can be characterized in terms $\mathcal{M}(\Omega)$ as well (see Section \ref{sec:Prelim}). 

\begin{thm}\label{thm:general} Assume \ref{itm:A1} and \ref{itm:A2}. Let $u_\lambda\in \mathrm{C}(\overline{\Omega}_\lambda)$ be the solution to \eqref{eq:S_lambda}. 
\begin{itemize}
    \item[$(\mathrm{i})$] We have $u_\lambda +  \lambda ^{-1}c_\Omega \to u^\gamma$ as $\lambda\to 0$ uniformly on $\overline{\Omega}$ and $u^\gamma$ is a solution to \eqref{eq:S_0}. 
    \item[$(\mathrm{ii})$] Furthermore $u^\gamma = \max_{w\in \mathcal{E}^\gamma} w$ where $\mathcal{E}^\gamma$ denotes the family of subsolutions $w$ of \eqref{eq:S_0} such that
    \begin{equation}\label{eq:maxM_gamma}
      \gamma\big\langle \mu,  (-x,v)\cdot \nabla L(x,v)\big\rangle_\Omega +\langle \mu, w\rangle_\Omega \leq 0
      \qquad\text{for all}\; \mu\in \mathcal{M}(\Omega).
    \end{equation}
    We note that $u^\gamma$ does not depend on $r(\lambda)$, but only on the limit $\gamma\in \mathbb{R}$ of $\lim_{\lambda\to 0^+} r(\lambda)/\lambda$. 
\end{itemize}
\end{thm}
As a consequence, if $r(\lambda)/\lambda \to 0$ as $\lambda\to 0^+$ then $u^\lambda + \lambda^{-1}c_\Omega \to u^0$, the same limiting solution to the vanishing discount problems on the fixed domain $\Omega$  as in \cite[Theorem 3.7]{ishii_vanishing_2017}.

\begin{rem} \quad 
\begin{itemize}
    \item[(i)] The term \emph{viscosity Mather measure} is introduced in \cite{ishii_vanishing_2017-1, ishii_vanishing_2017}, in constrast to \emph{stochastic Mather measure} introduced in \cite{gomes_stochastic_2002}. 
    We should denote the effect of viscosity by $\mathcal{M}_\varepsilon(\Omega)$ in \eqref{eq:M0}, but for simplicity, we did not include $\varepsilon$ since the results in this paper concern the limit as $\lambda\to 0^+$. Nevertheless, it is evident that weak limits of measures in $\mathcal{M}_\varepsilon(\Omega)$ correspond to measures in $\mathcal{M}_0(\Omega)$ due to the stability of viscosity solutions.
    \item[(ii)] In the proof of Theorem \ref{thm:limit} (where $\varepsilon>0$) and its corresponding first-order analog \cite{tu_vanishing_2021} (where $\varepsilon=0$), the scaling differs due to the appearance of $-\varepsilon\Delta$. We plan to address this discrepancy in an upcoming work, where we explore this issue together with the approach to this problem using the framework of \cite{gomes_stochastic_2002} instead. 
    % \item[(iii)] There a different way of defining Mather measures for second-order equations following \cite{gomes_stochastic_2002} can also be used to study the problem, which will appear in future work.
\end{itemize}
\end{rem}
The next result concerns the differentiability of $c_{\Omega_\lambda}$ with respect to the scaling parameter $r(\lambda)$. We observe that the choice of $r(\lambda)$ and $\gamma\in \mathbb{R}$ does not affect the study of the derivatives of $r(\lambda)\mapsto c_{(1+r(\lambda)\Omega)}$. To simplify notation, we assume the following.
\begin{description}[style=multiline, labelwidth=1cm, leftmargin=2.0cm]
    \item[\namedlabel{itm:A3}{$(\mathcal{A}_3)$}] We assume $r(\lambda) = \lambda$ for $\lambda \in (-\varepsilon_0,\varepsilon_0)$ for some $\varepsilon_0>0$. As $\Omega_\lambda = (1+\lambda)\Omega$, we will write $c(\lambda) = c_{\Omega_\lambda}$ and $c(0) = c_\Omega$ for simplicty. 
\end{description} 
It is important to note that in the context of \ref{itm:A3}, $\lambda$ does not represent a discount factor.

\begin{thm}\label{thm:limit} Assume \ref{itm:A1}, and \ref{itm:A3}. The map $\lambda\mapsto c(\lambda)$ is one-sided differentiable:
\begin{align}
\label{eq:limit1}
    c'_+(0) = &\lim_{\lambda\rightarrow 0^+} \left(\frac{c(\lambda) - c(0)}{\lambda}\right) = \max_{\mu\in \mathcal{M}(\Omega)} \left\langle \mu, (-x,v)\cdot \nabla L(x,v)\right\rangle_\Omega,\\
    \label{eq:limit2}
    c'_-(0) =&\lim_{\lambda\rightarrow 0^-} \left(\frac{c(\lambda) - c(0)}{\lambda}\right) = \min_{\mu\in \mathcal{M}(\Omega)} \left\langle \mu,(-x,v)\cdot \nabla L(x,v) \right\rangle_\Omega.
\end{align}
\end{thm}

A similar result is established for first-order Hamilton--Jacobi equations (see \cite[Theorem 1.4]{tu_vanishing_2021}).
Theorem \ref{thm:limit} gives us that the one-sided derivatives $c'_{\pm}(\lambda)$ exists \emph{everywhere} for $\lambda\in (-\varepsilon_0,\varepsilon_0)$. We emphasize that the set of $\lambda$ where $c'(\lambda)$ does not exist is \emph{at most countable}. We refer to \cite[Theorem 4.2, Chapter 4]{bruckner_differentiation_1978} or \cite[ Theorem 17.9]{hewitt_real_1965} for this result. %We summarize the properties of $\lambda\mapsto c(\lambda)$ under the assumption \ref{itm:A3} as follows.

\begin{cor}\label{cor:diff_ae} Assume \ref{itm:A1} and \ref{itm:A3}. 
\begin{itemize}
    \item[$\mathrm{(i)}$] The map $\lambda\mapsto c(\lambda)$ is Lipschitz, increasing, the one-sided derivatives $c'_+(\lambda)$ and $c'_-(\lambda)$ exist. 
    \item[$\mathrm{(ii)}$] $\lambda\mapsto c'_-(\lambda)$ is left-continuous and $\lambda\mapsto c'_+(\lambda)$ is right-continuous.
    \item[$\mathrm{(iii)}$] The set of points where $\lambda 
    \mapsto c(\lambda)$ is not differentiable is almost countable.
\end{itemize}
\end{cor}

\begin{rem}\label{rem:C(gamma)} By Theorem \ref{thm:general}, for each $\gamma\in \mathbb{R}$ there exists a solution $u^\gamma \in \mathrm{C}(\overline{\Omega})$ to \eqref{eq:S_0} and this solutition. In other words, we have a well-defined map $\Lambda: \mathbb{R}\to \mathrm{C}(\overline{\Omega})$. In contrast to the first-order case, solutions to \eqref{eq:S_0} are unique up to adding a constant (\cite{lasry_nonlinear_1989}). If $\gamma\in \mathbb{R}$ then there exists a unique constant, denoted by $\mathcal{C}(\gamma) \in \mathbb{R}$ such that $u^\gamma(x) - u^0(x) = \mathcal{C}(\gamma)$ for all $x\in \overline{\Omega}$. We therefore can define $\mathcal{C}(\gamma)$ as a function from $\mathbb{R}$ to $\mathbb{R}$ by $\mathcal{C}:\R\to \R$ by $\mathcal{C}(\gamma) := u^\gamma(\cdot) - u^0(\cdot) \in \R$.
\end{rem}

\begin{thm}\label{cor:back_forth} Assume \ref{itm:A1}, and that $c'(0)$ exists, in the sense that $\theta \mapsto c_{(1+\theta)\Omega}$ has derivative at $\theta = 0$. Then $\mathcal{C}:\mathbb{R}\to\mathbb{R}$ is differentiable at $0$ with $\mathcal{C}'(0) = -c'(0)$, and $\mathcal{C}(\gamma) = -\gamma c'(0)$ for $\gamma\in \R$. In other words, $u^\gamma(\cdot) = u^0(\cdot) - \gamma c'(0)$ for all $\gamma\in \mathbb{R}$. 
\end{thm}

Theorems \ref{thm:general}, \ref{thm:limit} and Corollary \ref{cor:diff_ae} extend similar findings from \cite{tu_vanishing_2021} for the first-order case. Details on the encountered challenges and their resolution are discussed in Section \ref{subsection:contributions}.

We gain further insight into the regularity of the map $\lambda \mapsto c(\lambda)$ with more details about $f(\cdot)$. Recall that a function $u:I \to \mathbb{R}$ is semiconvex on an interval $I \subset \mathbb{R}$ if $x \mapsto u(x) + (2\tau)^{-1}|x|^2$ is convex for some $\tau > 0$.

\begin{thm}\label{lem:example} Assume \ref{itm:A1}, \ref{itm:A3} and $f = \mathrm{const}$. Then $\lambda\mapsto c(\lambda)$ belongs to $\mathrm{C}^\infty$ in its domain.
\end{thm}

\begin{thm}\label{prop:semiconvex} Assume \ref{itm:A1}, \ref{itm:A3} and $f$ is semiconcave. Then $\lambda\mapsto c(\lambda)$ is semiconvex in its domain. As a consequence, it is twice differentiable almost everywhere.
\end{thm}

Assume \ref{itm:A1}, \ref{itm:A3} and $p=2$, the Hopf-Cole transform reveals that $\lambda\mapsto c(\lambda)$ is smooth, and $c'(0)$ has an integral representation (Section \ref{sec:the quad case}), known as Hadamard's variational formula (\cite[p. 369]{Evans2010}). 
This holds for general perturbations $\Omega_\lambda = {(\mathbf{Id}+\lambda \textbf{V})(x): x\in \Omega}$, where $\textbf{V}$ is a smooth vector field $\mathbb{R}^n\to \mathbb{R}^n$. Our scaling setting represents the special case where $\textbf{V} = \mathbf{Id}$.

Theorem 1.4 and findings in Section 6, such as Theorem 1.5 and Corollary 6.1, are new and exclusive to second-order scenarios.

\subsection{Contributions}\label{subsection:contributions} We mention our contributions here, mainly in two parts:
\begin{itemize}
    \item[(i)] The technical development in Theorems \ref{thm:general}, \ref{thm:limit} and Corollary \ref{cor:diff_ae}. 
    \item[(ii)] The new connection in Theorem \ref{cor:back_forth} and observations in Theorems \ref{lem:example} and \ref{prop:semiconvex}. 
\end{itemize}

For (i), the challenge arises from the new scaling with the second-order term, along with the absence of finite speed of propagation, in contrast to the first-order case. We address this by employing the tool developed in \cite{ishii_vanishing_2017-1, ishii_vanishing_2017} and carefully tracking all conditions where the representation formula can be applied. Another potential way to study this problem is using the stochastic analogue of Aubry-Mather theory from \cite{gomes_stochastic_2002}. The advantage of the duality framework in \cite{ishii_vanishing_2017-1, ishii_vanishing_2017} is its compatibility with state-constraint boundary conditions, whereas tools from \cite{gomes_stochastic_2002} require some technical adaptations for different boundary conditions. This aspect will be addressed in future work.
\begin{itemize}
    \item For the first-order case, using the finite speed of propagation, one can replace the space of the test function as $\Phi^+(\Omega) = \mathrm{C}(\overline{\Omega}\times \R^n)$ by $\Phi^+(\Omega) = \mathrm{C}(\overline{\Omega}\times \overline{B}_h)$ for some $h > 0$. 
    \item For the second-order case, the duality framework in \cite{ishii_vanishing_2017-1, ishii_vanishing_2017} works if one restrict the space of test functions to $\Phi^+(\Omega) = \left\{tL(x,v) + \chi(x): t>0, \chi\in \mathrm{C}(\overline{\Omega})\right\}$.
\end{itemize}
It turns out that the new scaling of the equation in the second-order case works perfectly with this new space of test functions, which results in the following differences:
\begin{itemize}
    \item We must assume that the Hamiltonian is separable, as we employ $H(x,\xi) = |\xi|^p - f(x)$ in this paper. The exploration of the general Hamiltonian will be pursued in future work.
    \item In contrast to the first-order case, the scaling structure is captured by $\nabla_{(x,v)} L$ instead of just the $x$-gradient $D_xL$ as seen in the first-order case (\cite{tu_vanishing_2021}). This can be roughly explained as a variation in the scaling structure of the equation and the utilization of the space of test functions. 
\end{itemize}

For (ii), this is a new result that is only available in the second-order case. The result also contains a connection to the problem of characterization of minimizing measures, i.e., how all minimizing measures can be approximated. 
%By employing the Diagram in Figure \ref{fig:my_diagram}, we systematically derived information that is necessary for the result. 
At the moment, we cannot show the differentiability of this map yet unless some special cases, as in Theorems \ref{lem:example} or the case where $p=2$ (Corollary \ref{thm:p=2}).

\begin{rem} According to the \cite{ishii_vanishing_2017-1, ishii_vanishing_2017}, the dependence of minimizing measures $\mathcal{M}(\Omega)$ on the choice of the space of test functions is unknown. This creates a difference in the limiting procedure as $\varepsilon$ approaches $0$ (in the first-order case) in representations \eqref{eq:limit1}, \eqref{eq:limit2}, \eqref{eq:maxM_gamma} when a different space of test functions is used (as in this paper and \cite{tu_vanishing_2021}). This aspect will be considered in future work. 
\end{rem}

\begin{rem} The problem of characterizing Mather measures has been studied, as in \cite{gomes_mather_2010}. One approach to proving their existence is through approximation from a related vanishing discount problem. A key question arises: can all viscosity Mather measures be described by $\bigcup_{\gamma\in \R} \mathcal{U}_\gamma(\Omega)$ (Definition \ref{def:solve_measures})?
\end{rem}

\subsection{Literature} 
A general form of \eqref{eq:main} is studied in \cite{ishii_vanishing_2017-1, ishii_vanishing_2017}, where the authors develop viscosity Mather Measures using duality. 
A different notion of viscosity Mather measures for second-order equation (in the periodic setting) defined using a stochastic analogue of the Aubry-Mather theory has been studied earlier in \cite{gomes_stochastic_2002}. 
In contrast to the first-order case, solutions to the ergodic problem \eqref{eq:S_0} are unique up to adding a constant (see \cite{lasry_nonlinear_1989}). This makes the vanishing discount problem on a fixed domain simpler to show, and it was shown in \cite{lasry_nonlinear_1989} using a pure PDE approach. The first-order case is harder, and it is a non-trivial problem to characterize the limiting solution (\cite{davini_convergence_2016, ishii_vanishing_2017-1, ishii_vanishing_2017, tran_hamilton-jacobi_2021}).

If $p>2$, there exists a unique bounded viscosity solution (\cite{armstrong_viscosity_2015, lasry_nonlinear_1989}) $u\in \mathrm{C}(\overline{\Omega})$ that is H\"older continuous, and it is a maximal solution in $W^{2,r}_{\mathrm{loc}}(\Omega)$ (for all $1<r<\infty$) of
\begin{equation}\label{eq:intro_eq}
    \delta u(x) + |Du(x)|^p - f(x) - \varepsilon \Delta u(x) = 0     \qquad\text{in}\;\Omega.
\end{equation}
When $1<p \leq 2$, equation \eqref{eq:main} can be written as (see \cite{lasry_nonlinear_1989}) \eqref{eq:intro_eq} coupled with $=+\infty$ on $\partial\Omega$, and solutions can be understood in the classical sense. When $p>2$, a solution is uniformly bounded up to $\partial\Omega$, and the equation is understood in the viscosity sense (\cite{armstrong_viscosity_2015, BARLES201031, lasry_nonlinear_1989}). We refer the readers to \cite{alessio_asymptotic_2006} (elliptic problem with subquadratic growth), \cite{attouchi_gradient_2020, barles_generalized_2004, barles_large_2010, tchamba_2010} (time-dependent problem), \cite{armstrong_viscosity_2015, lasry_nonlinear_1989} (elliptic problem) and the references therein for more properties of solutions and related problems. The rate of convergence as $\varepsilon\to 0$ is studied in \cite{YuTu2021large}.  

The vanishing discount problem on fixed domain concerns the convergence of solutions is considered in \cite{ishii_vanishing_2020} for the case, the domain is $\mathbb{R}^n$, \cite{davini_convergence_2016,ishii_vanishing_2017-1, mitake_selection_2017,tran_hamilton-jacobi_2021, zhang_limit_2022} for the case the domain is torus $\mathbb{T}^n = \mathbb{R}^n/ \mathbb{Z}^n$, and \cite{ishii_vanishing_2017} for a general domain with various boundary conditions. We refer the readers also to \cite{wang_convergence_2021} for a vanishing discount problem with more general nonlinear dependence on the monotone term. The case of changing domain for $\varepsilon=0$ is studied in \cite{tu_vanishing_2021}.

\subsection{Organization of the paper} The paper is organized in the following way. In Section \ref{sec:Prelim}, we provide the background results on state-constraint Hamilton--Jacobi equations, estimates on solutions, the duality representation, definitions, and properties of minimizing measures that will be needed throughout the paper. The proof of Theorem \ref{thm:general} is provided in Section \ref{sec:The vanishing discount problem}. Section \ref{sec:c'(0)} is devoted to proving the main result of Theorem \ref{thm:limit} and its consequences on getting the estimate on the nested domain. The connection between $c'(0)$ and $\mathcal{C}'(0)$ as in Theorem \ref{cor:back_forth} is discussed in Section \ref{sec:ergodic problem and the differentiability}. The last section records some discussions (Theorems \ref{lem:example}, \ref{prop:semiconvex}) and open questions, as well as future research related to the theme of the paper. In particular, we investigate a very special case where $p=2$ (it falls into the subquadratic case $1<p\leq 2$) using a tool from elliptic theory, which provides a different perspective on the derivative $c'(0)$ of the additive eigenvalue.

\section{Preliminaries}\label{sec:Prelim}
We say $u$ solves \eqref{eq:reference_problem} if it solves the state-constraint problem on $\mathcal{O}$ with discount factor $\delta\geq 0$ as in Definition \ref{defn:solve}. We make the assumption regarding the geometry of $\Omega$ as specified in \ref{itm:A1}.
% \subsection{Notations}
% \begin{itemize}
%     % \item $I = [-\varepsilon_0, \varepsilon_0] \subset \R$ is the domain of the variation parameter for the case $r(\lambda) = \lambda$. 
%     \item 
%     % \item $\mathcal{M}(\Omega)$ is the set of probability viscosity Mather measures on $\Omega$ (Definition \ref{def:M_0}).
%     % \item $\mathcal{U}_\gamma(\Omega)$ is the set of viscosity Mather measures that can be approximated by the vanishing discount with changing domain where $\gamma$ is defined as in \eqref{eq:gam}.
% \end{itemize}

%Without loss of generality, we will always assume $0\in \Omega$. We assume that $\gamma$ from \eqref{eq:gam} is finite, $H(x,\xi) = |\xi|^p - f(x)$ where $p>2$, $f\in \mathrm{C}^1(\overline{\Omega})$ and $U\subset \mathbb{R}^n$ is an open, bounded set with smooth boundary containing $\Omega$. We make the assumption regarding the geometry of $\Omega$ as specified in \ref{itm:A1}.

%In view of the notation of integration against a measure in \eqref{eq:integration_notation_main}, if the underlying domain is understood, we simply write $\langle \mu, \varphi\rangle$.

\subsection{State-constraint solutions and the representation formula}
Proofs for certain results in this section, as well as background information on viscosity solutions, can be found in Appendix~\ref{ap:a}.

\begin{defn}\label{defn:0} We say that $u$ is a viscosity state-constraint solution of \eqref{eq:main} if $u$ is a subsolution of the equation in $\Omega$ and is a supersolution of the equation on $\overline{\Omega}$.
\end{defn}

The state-constraint problem was first studied in \cite{soner_optimal_1986-1} for the first-order equation using the \emph{interior sphere condition}. Assumption \ref{itm:A1} is introduced in \cite{capuzzo-dolcetta_hamiltonjacobi_1990} as a simpler way to obtain uniqueness.
The well-posedness for \eqref{eq:main} and the properties of its solution can be summarized as follows. 

\begin{thm}\label{thm:Holder_main} Assume \ref{itm:A1} and $\lambda\geq 0$. If $u$ is a viscosity solution to \eqref{eq:main} then $u \in \mathrm{C}^{0,\alpha}(\overline{\Omega})$ where $\alpha = (p-2)/(p-1)$. Furthermore:
\begin{itemize}
    \item[(i)] If $\lambda > 0$ then there exists a unique continuous viscosity solution $u$, and it is the maximal subsolution of $\lambda u + H(x,Du) - \varepsilon \Delta u = 0$ in $\Omega$ in the class of upper semicontinuous subsolution.
    \item[(ii)] If $\lambda = 0$, then up to adding a constant, there exists a unique continuous viscosity solution. 
\end{itemize}
\end{thm}

\begin{thm}[Comparison principle]\label{thm:comparison_main} Assume \ref{itm:A1} and $\lambda>0$. Let $u,v\in \mathrm{C}(\overline{\Omega})$ be a viscosity subsolution in $\Omega$ and supersolution on $\overline{\Omega}$ of $\lambda u + H(x,Du) - \varepsilon \Delta u = 0$, respectively. Assume that $u,v$ are H\"older continuous, i.e., belong to $\mathrm{C}^{0,\alpha}(\overline{\Omega})$ for some $\alpha \in (0,1]$, then $u(x)\leq v(x)$ for all $x\in \overline{\Omega}$.
\end{thm}

An important consequence is the following bound (consequence from 
Proposition \ref{pro:bound_on_sln}).

\begin{prop}\label{pro:bound_on_sln_main}
Assume \ref{itm:A1} and $\lambda > 0$. The solution $u_\lambda$ of \eqref{eq:main} satisfies 
\begin{itemize}
    \item[(i)] $-\max_{x\in \overline{\Omega}}H(x,0) \leq \lambda u_\lambda(x)$ for $x\in \overline{\Omega}$,
    \item[(ii)] if $B(0,\kappa) \subset \Omega$ then $\lambda u_\lambda(x) \leq C(\kappa)$ for $x\in \overline{B(0,\kappa)}$, where $C(\kappa)$ depends only on $\kappa$. 
\end{itemize}
\end{prop}

We employ the representations in \cite{ishii_vanishing_2017-1, ishii_vanishing_2017} (see also \cite[Theorem 6.3]{tran_hamilton-jacobi_2021}). We note that a different representation not using duality is given in \cite{gomes_stochastic_2002} for the periodic setting. % and the references therein). 
%More details on the representation within a general setting is given in Appendix. 

We denote by $\mathcal{P}$ the set of all probability measures, and likewise, if $E$ is a metric space, we denote by $\mathcal{P}_E$ the set of all probability measures on $E$. The set of all Radon measures (resp., nonnegative Radon measures) in $E$ is denoted by $\mathcal{R}(E)$ (resp., $\mathcal{R}^+(E)$).

\begin{defn}\label{def:cone_main} Let us define 
\begin{equation*}
    \Phi^+(\overline{\Omega}\times \R^n) := \left\lbrace \phi\in \mathrm{C}(\overline{\Omega}\times \R^n): \phi(x,v) = tL(x,v) + \chi(x), t>0, \chi\in \mathrm{C}(\overline{\Omega})\right\rbrace.
\end{equation*}
For each $\phi\in \Phi^+(\overline{\Omega}\times\R^n)$, define $H_\phi(x,\xi) = \sup_{v\in \mathbb{R}^n}\left(\xi\cdot v - \phi(x,v)\right)$ for $(x,\xi)\in \overline{\Omega}\times\R^n$. For $\delta\geq 0$ and $z\in \overline{\Omega}$ we define
\begin{align*}
    \mathcal{F}_{\delta,\Omega} 
        &= \Big\lbrace (\phi,u) \in \Phi^+(\overline{\Omega}\times\R^n)\times \mathrm{C}(\overline{\Omega}): \delta u + H_\phi(x,Du) - \varepsilon \Delta u \leq 0\;\text{in}\;\Omega \Big\rbrace,\\
    \mathcal{G}_{z,\delta,\Omega} 
        &= \Big\lbrace \phi - \delta u(z): (\phi, u)\in \mathcal{F}_{\delta,\Omega} \Big\rbrace,\\
    \mathcal{G}'_{z,\delta,\Omega} 
        &= \Big\lbrace \mu\in \mathcal{R}(\overline{\Omega}\times \R^n): \langle \mu, g\rangle_\Omega \geq 0\;\text{for all}\;g\in \mathcal{G}_{z,\delta,\Omega} \Big\rbrace.
\end{align*}
When $\delta = 0$, we note that $\mathcal{G}_{z,0,\Omega}$ and $\mathcal{G}'_{z,0,\Omega}$ are independent of $z\in \overline{\Omega}$. Therefore we denote 
\begin{equation*}
    \mathcal{G}_{0,\Omega}:=\mathcal{G}_{z,0,\Omega}
    \qquad\text{and}
    \qquad 
    \mathcal{G}'_{0,\Omega}:=
    \mathcal{G}'_{z,0,\Omega}\qquad \text{for all}\;z\in \overline{\Omega}.
\end{equation*}
Precisely, we have 
\begin{align*}
    \mathcal{G}_{0,\Omega} 
    &= \Big\lbrace \phi\in \Phi^+(\overline{\Omega}\times \mathbb{R}^n):\;\text{there exists}\;u\in \mathrm{C}(\overline{\Omega}): H_\phi(x, Du) - \varepsilon \Delta u \leq 0\;\text{in}\;\Omega \Big\rbrace,\\
    \mathcal{G}'_{0,\Omega} 
    &= \Big\lbrace \mu \in \mathcal{R}(\overline{\Omega}\times\mathbb{R}^n): \langle \mu, g\rangle_\Omega\geq 0\;\text{for all}\;g\in \mathcal{G}_{0,\Omega}\Big\rbrace.
\end{align*}
\end{defn}
We observe that $\Phi^+(\overline{\Omega}\times \R^n) $ is a convex cone in $\mathrm{C}(\overline{\Omega}\times \R^n)$ and $(x,\xi)\mapsto H_\phi(x,\xi)$ is well-defined and continuous for $\phi\in \Phi^+(\overline{\Omega}\times \R^n)$.

\begin{thm}[{\cite[Theorem 3.3]{ishii_vanishing_2017}}]\label{thm:repdelta>0_main} Assume \ref{itm:A1}. Let $(z,\lambda) \in \overline{\Omega}\times (0,\infty)$ and $u_\lambda\in \mathrm{C}(\overline{\Omega})$ be a solution of \eqref{eq:main}. Then for $\lambda > 0$ there holds
\begin{align}
    \lambda u_\lambda(z) &= \min_{\mu\in \mathcal{P}\cap \mathcal{G}'_{z,\lambda,\Omega}}  \langle\mu, L\rangle_\Omega \label{eq:representation_main-lambda} \\
    -c_\Omega &= \min_{\mu\in \mathcal{P}\cap \mathcal{G}'_{0, \Omega}} \langle \mu, L\rangle_\Omega. \label{eq:representation_main-0}
\end{align}
\end{thm}

\begin{defn}[Viscosity Mather measures, {\cite{ishii_vanishing_2017-1, ishii_vanishing_2017}}]\label{def:M_0} We define 
    \begin{equation}\label{eq:M0}
        \mathcal{M}(\Omega): = \left\lbrace \mu \in \mathcal{P}\cap \mathcal{G}'_{0,\Omega}: -c_\Omega = \langle\mu, L\rangle_\Omega \right\rbrace.
    \end{equation}
\end{defn}

% \begin{rem}\label{lem:using-measures-shortcut} We will use Definition \ref{def:cone_main} mostly in the following way: if $\phi \in \Phi^+(\overline{\Omega}\times\mathbb{R}^n)$ then whenever $\delta\geq 0$, $z\in \overline{\Omega}$ and $u\in \mathrm{C}(\overline{\Omega})$ such that
% \begin{align*}
%     \delta u + H_\phi(x,Du) - \varepsilon\Delta u \leq 0 \;\text{in}\;\Omega \qquad\Longrightarrow \qquad \langle \mu_\delta, \phi-\delta u(z)\rangle_\Omega \geq 0\qquad \text{for all}\;\mu_\delta \in \mathcal{G}'_{z,\delta,\Omega}.
% \end{align*}
% \end{rem}
We will use Definition \ref{def:cone_main} mostly in the following way. 
\begin{cor}\label{cor:usage-measures} Let $\phi \in \Phi^+(\overline{\Omega}\times\mathbb{R}^n)$ and $\delta\geq 0$, $z\in \overline{\Omega}$ and $u\in \mathrm{C}(\overline{\Omega})$ such that 
\begin{equation*}
    \delta u + H_\phi(x,Du) - \varepsilon\Delta u \leq 0 \;\text{in}\;\Omega .
\end{equation*}
\begin{itemize}
    \item[$\mathrm{(i)}$] If $\mu_\delta \in \mathcal{G}'_{z,\delta,\Omega} $ then $\langle \mu_\delta, \phi-\delta u(z)\rangle_\Omega \geq 0$. 
    \item[$\mathrm{(ii)}$] If $\mu \in \mathcal{G}'_{0,\Omega}$ then $ \langle \mu, \phi-\delta u\rangle_\Omega \geq 0$.
\end{itemize}
\end{cor}
\begin{proof} The conclusion follows from $H_{\phi-\delta u}(x,Du) - \varepsilon \Delta u = \delta u + H_\phi(x,Du) - \varepsilon \Delta u \leq 0$.
\end{proof}

\subsection{Settings for problems on changing domains}

\begin{defn}\label{defn:solve} For an open subset $\mathcal{O}\subset \R^n$ and $\delta>0$, we say $u\in \mathrm{C}(\mathcal{O})$ solves \eqref{eq:reference_problem} if $u$ solves 
\begin{equation}\label{eq:reference_problem}\tag{$\delta, \mathcal{O}$}
    \begin{cases}
        \delta u(x) + |Du(x)|^p - f(x) - \varepsilon\Delta u(x) \leq 0 \qquad\text{in}\;\mathcal{O},\\
        \delta u(x) + |Du(x)|^p - f(x) - \varepsilon\Delta u(x) \geq 0 \qquad\text{on}\;\overline{\mathcal{O}}
    \end{cases}
\end{equation}
If $\delta = 0$, we say $u\in \mathrm{C}(\mathcal{O})$ solves \eqref{eq:reference_problem_ergodic} if $u$ solves 
\begin{equation}\label{eq:reference_problem_ergodic}\tag{$0, \mathcal{O}$}
    \begin{cases}
        |Du(x)|^p - f(x) - \varepsilon\Delta u(x) \leq c(\mathcal{O}) \qquad\text{in}\;\mathcal{O},\\
        |Du(x)|^p - f(x) - \varepsilon\Delta u(x) \geq c(\mathcal{O}) \qquad\text{on}\;\overline{\mathcal{O}}
    \end{cases}
\end{equation}
where $c(\mathcal{O})$ is the additive eigenvalue on $\mathcal{O}$, as defined in \eqref{eq:c(0)}.
\end{defn}

\begin{lem}\label{lem:bound_delta_u_delta_eig} Let $\mathcal{O}\subset \mathbb{R}^n$ be a bounded domain, and assume \ref{itm:A1} for $\mathcal{O}$. 
\begin{itemize}
    \item[$\mathrm{(i)}$] For $\delta>0$, the unique solution $u_\delta$ to \eqref{eq:reference_problem} satisfies 
    \begin{equation*}
        \left|\delta u_\delta(x) + c(\mathcal{O})\right| \leq C_2\delta \qquad\text{for all}\;\delta>0, x\in \overline{\mathcal{O}},
    \end{equation*}
    where $C_2$ depends only on $\mathrm{diam}(\mathcal{O})$, $\Vert f\Vert_{L^\infty(\mathcal{O})}$ and $p$.
    \item[$\mathrm{(ii)}$] There holds $|c(\mathcal{O})|\leq C$ where $C$ depends only on $\min_{\overline{\mathcal{O}}}f$ on $U$ and $\kappa$ as defined in \ref{itm:A1}.
\end{itemize}
\end{lem}
\begin{proof} Thanks to the H\"older estimate of $u_\delta$ in Theorem \ref{thm:Holder_main}, we have 
\begin{equation*}
    u_\delta(x) - u_\delta(0) \to w(x), \qquad \delta u_\delta(x)\to -c(\mathcal{O})
\end{equation*}
uniformly on $\mathcal{O}$ and $w$ solves the ergodic problem \eqref{eq:reference_problem_ergodic} with the bound $|w(x)|\leq C|x|^\alpha \leq C$, for all $x\in \mathcal{O}$ where $C=C\big(\mathrm{diam}(\mathcal{O}),\Vert f\Vert_{L^\infty(\mathcal{O})},p\big)$. Now $w(\cdot)+\delta^{-1}c(\mathcal{O})$ and $w(\cdot)-\delta^{-1}c(\mathcal{O})$ are supersolution and subsolution to \eqref{eq:reference_problem}, respectively. By comparison principle we deduce that $\left|\delta u_\delta(\cdot) + c(\mathcal{O})\right|\leq \delta \Vert w \Vert_{L^\infty(\mathcal{O})}$. 

For (ii), let $u_\delta\in \mathrm{C}(\mathcal{O})$ solves \eqref{eq:reference_problem}, i.e., the state-constraint problem with discount factor $\delta$ on $\mathcal{O}$. By Proposition \ref{pro:bound_on_sln_main}, $|\delta u_\delta(0)| \leq C$ where $C$ depends only on $\min \{f(x):x\in \overline{U}\}$ and $\kappa$. Let $\delta \to 0^+$ we deduce that $|c(\mathcal{O})|\leq C$.
\end{proof}

We obtain an improved bound for $\delta u_\delta(\cdot)$ that for solution $u_\delta(\cdot)$ of \eqref{eq:reference_problem}.
\begin{cor}\label{cor:improved_bounds} Let $\mathcal{O}\subset \mathbb{R}^n$ be a bounded domain, and assume \ref{itm:A1} for $\mathcal{O}$. There exists a constant $C$ that depends only on $\min_{\overline{U}} f$ and $\kappa$ such that:
\begin{itemize}
    \item[(i)] If $|r(\lambda)|$ is small so that $B(0,\kappa)\subset \Omega_\lambda$ then $|c_{\Omega_\lambda}|\leq C$.
    \item[(ii)] If $u_\delta(\cdot)$ solves \eqref{eq:reference_problem} then $|\delta u_\delta(\cdot)| \leq C+C_2\delta$, where $C_2$ is defined in Lemma \ref{lem:bound_delta_u_delta_eig}.
\end{itemize}
\end{cor}

The error estimate for discount solutions on slightly different domains implies that the map $\lambda\mapsto c(\lambda)$ is H\"older continuous with degree $\alpha = (p-2)/(p-1)$. However, it could be of interest on its own, as seen in the setting of \cite{kim_state-constraint_2020}. We provide the proof in Appendix~\ref{ap:c} for completeness.

\begin{thm}[Nested domains estimate]\label{thm:Holder_est_nested_domain} Let $\mathcal{O}\subset \mathbb{R}^n$ be a bounded domain, and assume \ref{itm:A1} for $\mathcal{O}$. Let $\theta >0$, $0<\delta < 1$ and $u \in \mathrm{C}((1+\theta)\overline{\Omega})$ and $v\in \mathrm{C}(\overline{\Omega})$ be solutions to $(\delta, (1+\theta)\Omega)$ and $(\delta, \Omega)$, respectively. Then, it holds that
\begin{equation*}
    0 \leq \delta  v(x) -  \delta u(x) \leq C\theta ^\alpha, \qquad \text{for all}\; x\in \overline{\Omega}.
\end{equation*}
where $\alpha = (p-2)/(p-1)$ and $C$ depends on $\mathrm{diam}(\Omega)$, the Lipschitz constant of $f$ on $U$ and $p$.
\end{thm}

\begin{cor}\label{cor:Holder_c(lambda)} Assume \ref{itm:A1}. If $\Omega' = (1+\theta)\Omega$ then 
\begin{equation*}
    |c_{\Omega'} - c_{\Omega}| \leq C\left(\lambda + |\theta|^\alpha\right) \qquad\text{where}\; \alpha = (p-2)/(p-1).
\end{equation*}
Consequently, under \ref{itm:A3} then $\lambda\mapsto c(\lambda)$ is H\"older continuous of degree $\alpha$.
\end{cor}

\begin{proof} Let $u_\lambda\in \mathrm{C}(\overline{\Omega'})$ be a solution to $(\lambda, \Omega')$ and $v_\lambda\in \mathrm{C}(\overline{\Omega})$ be a solution to \eqref{eq:main}, respectively. From Theorem \ref{thm:Holder_est_nested_domain} and Lemma \ref{lem:bound_delta_u_delta_eig} we obtain
\begin{equation*}
\begin{aligned}
    |c_{\Omega'} - c_{\Omega}| 
        &\leq | \lambda  u_\lambda(0) + c_{\Omega'}| + | \lambda v_\lambda(0) + c_\Omega| +  \lambda |u_\lambda(0) - v_\lambda(0)| \leq 2C \lambda  + C|\theta|^\alpha.
\end{aligned}
\end{equation*}
Therefore $\lambda\mapsto c(\lambda)$ is H\"older continuois of degree $\alpha$ if we assume \ref{itm:A3}.
% If $\gamma\neq 0$ then using $r(\lambda)/ \lambda \to \gamma\in \R$ as $\lambda \to 0$, we obtain $\lambda \leq C|r(\lambda)|$, thus the result follows.
\end{proof}

We can improve that $\lambda\mapsto c(\lambda)$ is Lipschitz in Theorem \ref{thm:limit}.

\subsection{Limit of minimizing measures}

\begin{defn}\label{defn:scaledown} Let $\sigma$ be a measure on $(1+r)\overline{\Omega}\times \R^n$, we define $\tilde{\sigma 
}$ as a measure on $\overline{\Omega}\times \R^n$ by
\begin{equation}\label{def_measures}
    \int_{\overline{\Omega}\times\R^n} \varphi(x,v)\;d\tilde{\sigma}(x,v) = \int_{(1+r)\overline{\Omega}\times\R^n} \varphi\left(\frac{x}{1+r},v\right)\;d\sigma(x,v), \qquad \varphi\in \mathrm{C}(\overline{\Omega}\times\R^n).
\end{equation}
\end{defn}

Using $\varphi\equiv 1$ we see that this scaling is mass-preserving. Recall that $\Omega_\lambda = (1+r(\lambda))\Omega$.

\begin{lem}\label{lem:u_gamma} Assume \ref{itm:A1} and \ref{itm:A2}. 
\begin{itemize}
    \item[(i)] Let $\{\mu_\lambda\}$ be a sequence of measures in $\mathcal{M}( \lambda , \Omega_\lambda)$. Let $\{\tilde{\mu}_\lambda\}$ be its scaling defined on $\overline{\Omega}\times\R^n$ and $\mu_\lambda\rightharpoonup \mu$ along a subsequence to some measure $\mu$, then $\mu\in \mathcal{M}(\Omega)$.
    \item[(ii)] Let $\{\nu_\lambda\}$ be a sequence of measures in $\mathcal{M}(0, \Omega_\lambda)$. Let $\{\tilde{\mu}_\lambda\}$ be its scaling defined on $\overline{\Omega}\times\R^n$ and $\nu_\lambda\rightharpoonup \mu$ along a subsequence to some measure $\nu$, then $\mu\in \mathcal{M}(\Omega)$.
\end{itemize}
\end{lem}

The proof of this Lemma is obmitted as is quite standard (adaptation of \cite[Lemma 3.6]{tu_vanishing_2021}).

\section{The vanishing discount problem} \label{sec:The vanishing discount problem}

Let $u_\lambda\in \mathrm{C}(\overline{\Omega}_\lambda)$ be the solution to \eqref{eq:S_lambda}. In this section, we show $u_\lambda + \lambda^{-1}c_\Omega \to u^\gamma$ as $\lambda \to 0^+$. We scale $u_\lambda$ into $\tilde{u}_\lambda \in \mathrm{C}(\overline{\Omega})$ for convenience by
\begin{equation}\label{eq:utilde}
    \tilde{u}_\lambda(x) = \frac{u_\lambda \big(\left(1+r(\lambda)\right)x\big)}{\left(1+r(\lambda)\right)^2}, \qquad x\in \overline{\Omega}.
\end{equation}
We then have:
\begin{equation}\label{eq:equation_of_utilde}
\begin{aligned}
    & \lambda \left(1+r(\lambda)\right)^2 \tilde{u}_\lambda(x) + H\Big(\left(1+r(\lambda)\right)x, \left(1+r(\lambda)\right) D\tilde{u}_\lambda(x)\Big)
    - \varepsilon \Delta \tilde{u}_\lambda(x) \leq 0\; \text{in}\;\Omega.
\end{aligned}
\end{equation}
The main idea is scaling equation from a domain to a different domain while maintaining the following (equivalent) forms
\begin{equation}\label{eq:form-1}
    \delta u + H_\phi(x,Du) - \varepsilon \Delta u \leq  0 \qquad\text{in}\;\mathcal{O} \qquad\text{or}\qquad  H_{\phi-\delta u}(x,Du) - \varepsilon \Delta u \leq  0 \qquad\text{in}\;\mathcal{O},
\end{equation}
where $\mathcal{O} = \Omega$ or $\Omega_\lambda$, and then applying Corollary \ref{cor:usage-measures} in an appropriate way. More importantly, for $\alpha, \beta \in \mathbb{R}\backslash \{0\}$ we have
\begin{align}\label{eq:form-2}
    H(\alpha x, \beta Du) = H_{\phi}(x,Du) \qquad\text{where}\qquad \phi(x,v) = L\left(\alpha x, \beta^{-1}v\right).
\end{align}
Indeed, we have 
\begin{equation*}
\begin{aligned}
     H_{\phi}(x,Du(x)) 
     &= \sup_{v} \Big(v\cdot Du(x) - L\left(\alpha x, \beta^{-1}v\right)\Big)\\
     &= \sup_{\beta \hat{v}} \Big(\hat{v}\cdot \beta Du(x) - L\left(\alpha x, \hat{v}\right)\Big) \\
     &= \sup_{\hat{v}} \Big(\hat{v}\cdot \beta Du(x) - L\left(\alpha x, \hat{v}\right)\Big) = H(\alpha x, \beta Du(x))
\end{aligned}
\end{equation*}
due to Legendre transform.

\begin{lem}\label{lem:bdd_norm} Assume \ref{itm:A1} and \ref{itm:A2}. If $u_\lambda \in \mathrm{C}(\overline{\Omega}_\lambda)$ solves \eqref{eq:S_lambda} then,
\begin{equation*}
        | \lambda  u_\lambda(x) + c_{\Omega_\lambda}| \leq C \lambda  \qquad\text{and}\qquad 
        | \lambda  u_\lambda(x) + c_\Omega| \leq C( \lambda  + |r(\lambda)|),
\end{equation*}
for all $x\in \overline{\Omega}_\lambda$, where $C$ is independent of $\lambda$. 
\end{lem}
\begin{proof} The first inequality is a consequence of Lemma \ref{lem:bound_delta_u_delta_eig}. The second inequality is the consequence of the first inequality and $|c_{\Omega_\lambda} - c_\Omega|\leq C|r(\lambda)|$ as in Corollary \ref{cor:diff_ae}.
\end{proof}

We use the normalization $\tilde{u}_\lambda(\cdot) +  \lambda^{-1}c_\Omega$ to make it consistent with the first-order case \cite{tu_vanishing_2021}. 

\begin{proof}[Proof of Theorem \ref{thm:general}] We split the proof into 3 steps. In steps 1 and 2, we establish two inequalities that, when used together, yield the limit of $\tilde{u}_\lambda + \lambda^{-1}c_\Omega \to u^\gamma$. In step 3, we give a characterization of $u^\gamma$. We start by observing that the sequence 
\begin{equation}\label{eq:normalize}
    \tilde{u}_\lambda(\cdot) + \frac{c_\Omega}{\lambda}    
\end{equation}
is bounded and equicontinuous, thus it has convergence subsequence. We show that such a limit is unique. Assume that there exist $\lambda_j\rightarrow 0$ and $\delta_j \rightarrow 0$ such that
\begin{equation}\label{eq:vanishing-assume-two-limits}
    \lim_{\lambda_j\to 0}\left(\tilde{u}_{\lambda_j}(x)+\frac{c_\Omega}{\lambda_j}\right)=u(x) \qquad\text{and}\qquad \lim_{\delta_j\to 0}\left(\tilde{u}_{\delta_j}(x)+\frac{c_\Omega}{\delta_j}\right)= w(x)
\end{equation}
locally uniformly as $j\to \infty$. For given $z\in \overline{\Omega}$, we show that $u(z) = w(z)$. We will show the following inequalities,
\begin{align}
    &\gamma \left\langle\mu, (-x, v)\cdot \nabla L(x,v)\right\rangle_\Omega + \langle \mu, w\rangle_\Omega - 2\gamma c_\Omega \leq 0, &&\forall \,\mu\in \mathcal{M}(\Omega),\label{eq:scale_down-1} \\
    &\gamma\big\langle \mu_0, (-x,v)\cdot \nabla L(x,v)\big\rangle_\Omega + \langle \mu_0, w\rangle_\Omega-2\gamma c_\Omega  + u(z) - w(z)  \geq 0, &&\text{for some} \, \mu_0\in \mathcal{M}(\Omega). \label{eq:scale_down-2}
\end{align}
Then, consequently $u(z)\geq w(z)$. We can reverse the roles of $u(z)$ and $w(z)$ to obtain $u(z)\leq w(z)$ and thus $u(z)\equiv w(z)$. \medskip

\paragraph{\textbf{Step 1.}} We show \eqref{eq:scale_down-1}. \medskip

\noindent
% To show \eqref{eq:scale_down-1}, 
We start with equation of $\tilde{u}_\lambda$ as in \eqref{eq:equation_of_utilde}. Using the strategy in \eqref{eq:form-1}, \eqref{eq:form-2} we obtain
\begin{equation*}
    \begin{cases}
    \begin{aligned}
        &H_\phi(x,D\tilde{u}_\lambda(x)) - \varepsilon \,\Delta
    \tilde{u}_\lambda(x)\leq 0 \qquad  \text{in}\;\Omega\\
        & \phi(x,v) = L\left(\left(1+r(\lambda)\right)x, \frac{v}{1+r(\lambda)}\right)  -  \lambda \left(1+r(\lambda)\right)^2 \tilde{u}_\lambda(x) \in \Phi^+(\overline{\Omega}\times \R^n).
    \end{aligned}
    \end{cases}
\end{equation*}
Using Corollary \ref{cor:usage-measures} (ii) and definition of $\mathcal{M}(\Omega)$ we have
\begin{equation*}
\begin{aligned}
    & \langle \mu, \phi\rangle_{\Omega} \geq 0 &&\qquad \mu\in \mathcal{M}(\Omega) \\
    -&\langle \mu, L\rangle_{\Omega} = c_\Omega && \qquad
    \mu\in \mathcal{M}(\Omega).
\end{aligned}
\end{equation*}
Therefore
\begin{align*}
    \left\langle \mu,  L\left(\left(1+r(\lambda)\right)x, \frac{v}{1+r(\lambda)}\right) -  \lambda \left(1+r(\lambda)\right)^2 \tilde{u}_\lambda(x) - L(x,v)\right\rangle_\Omega \geq c_\Omega \qquad\forall\;\mu\in \mathcal{M}(\Omega).
\end{align*}
We rearrange the left-hand side to match the form of \eqref{eq:normalize}:
\begin{align*}
    &\left\langle \mu,  L\left(\left(1+r(\lambda)\right)x, \frac{v}{1+r(\lambda)}\right)  - L(x,v)\right\rangle_\Omega 
    \geq 
    \left\langle\mu, \lambda \left(1+r(\lambda)\right)^2 \tilde{u}_\lambda(x) + c_\Omega\right\rangle_\Omega
    \\
    &\qquad\qquad\qquad\qquad = \lambda (1+r(\lambda))^2 \left\langle \mu, \tilde{\mu}_\lambda + \frac{c_\Omega}{\lambda}\right\rangle_\Omega + \left(1-(1+r(\lambda))^2\right) c_\Omega, \qquad\forall\;\mu\in \mathcal{M}(\Omega).
\end{align*}
Divide both sides by $\lambda$ and let $\lambda = \delta_j\to 0^+$, by \eqref{eq:vanishing-assume-two-limits} we obtain
\begin{align*}
    -\gamma \left\langle\mu, (-x, v)\cdot \nabla L(x,v)\right\rangle_\Omega \geq \langle \mu, w\rangle_\Omega -2\gamma c_\Omega, \qquad\forall\;\mu\in \mathcal{M}(\Omega).
\end{align*}
Therefore \eqref{eq:scale_down-1} follows.\medskip

\paragraph{\textbf{Step 2.}} We show \eqref{eq:scale_down-2}. \medskip

\noindent
We have
\begin{equation*}
    H(x,Dw(x))-\varepsilon\Delta w(x)\leq c_\Omega, \qquad \text{in}\;\Omega.
\end{equation*}
We scale this equation into an equation on $\Omega_{\lambda}$ as follows: 
\begin{equation*}
    \begin{cases}
    \begin{aligned}
        &\tilde{w}_\lambda(x) = \left(1+r(\lambda)\right)^2w\left(\frac{x}{1+r(\lambda)}\right) &&\text{on}\;\overline{\Omega}_\lambda,\\ 
        &H\left(\frac{x}{1+r(\lambda)}, \frac{1}{1+r(\lambda)} D\tilde{w}_\lambda(x)\right) - \varepsilon \Delta \tilde{w}_\lambda(x) \leq c_\Omega \qquad  &&\text{in}\;\Omega_\lambda.
    \end{aligned}
    \end{cases}
\end{equation*}
Using the strategy in \eqref{eq:form-1}, \eqref{eq:form-2} we obtain
\begin{equation*}
    \begin{cases}
    \begin{aligned}
        &\lambda \tilde{w}_\lambda(x) + H_\psi(x,D\tilde{w}_\lambda(x)) - \varepsilon \,\Delta
    \tilde{w}_\lambda(x)\leq 0 \qquad  \text{in}\;\Omega_\lambda\\
        & \psi(x,v) = L\left(\frac{x}{1+r(\lambda)}, \left(1+r(\lambda)\right)v\right) +\lambda \tilde{w}_\lambda(x) + c_\Omega\in \Phi^+\left(\overline{\Omega}_\lambda\times\R^n\right)
    \end{aligned}
    \end{cases}
\end{equation*}
Using Corollary \ref{cor:usage-measures} (i) we have
\begin{equation}\label{eq:convergence-vanishing-abc1}
    \langle \sigma, \psi - \lambda \tilde{w}_\lambda(z)\rangle_{\Omega_\lambda} \geq 0, \qquad \sigma\in \mathcal{P}\cap\mathcal{G}'_{z,\lambda,\Omega_\lambda}. 
\end{equation}
For all $\lambda>0$, by Theorem \ref{thm:repdelta>0_main} there exist $\mu_{\lambda} \in \mathcal{P}\cap \mathcal{G}'_{z, \lambda ,\Omega_{\lambda}}$ such that 
\begin{equation}\label{eq:convergence-vanishing-abc2}
    \langle \mu_\lambda, L \rangle_{\Omega_\lambda}=\lambda u_\lambda(z).
\end{equation}
In \eqref{eq:convergence-vanishing-abc1}, we set $\sigma = \mu_\lambda$, and by combining it with \eqref{eq:convergence-vanishing-abc2}, we obtain:
\begin{equation}\label{eq:convergence-vanishing-abc3}
    \left\langle \mu_\lambda, L\left(\frac{x}{1+r(\lambda)}, \left(1+r(\lambda)\right)v\right) - L(x,v) + \lambda u_\lambda(z) -\lambda \tilde{w}_\lambda(z) + c(0)  + \lambda \tilde{w}_\lambda\right\rangle_{\Omega_\lambda} \geq 0.
\end{equation}
Let $\tilde{\mu}_\lambda$ be the measure obtained from $\mu_\lambda$ defined as in Definition \ref{defn:scaledown}.  We write \eqref{eq:convergence-vanishing-abc3} as 
\begin{align*}
    \Big\langle \tilde\mu_\lambda, L\left(x, \left(1+r(\lambda)\right)v\right) - L\big((1+r(\lambda))x,v\big)\Big\rangle_\Omega  
    &+ \lambda u_\lambda(z) -\lambda \tilde{w}_\lambda(z) + c_\Omega \\
    &+ \lambda(1+r(\lambda))^2\langle \tilde\mu_\lambda, w\rangle_\Omega   \geq 0.
\end{align*}
Divide both sides by $\lambda>0$ we have
\begin{align}\label{eq:convergence-vanishing-abc4}
    \left\langle \tilde\mu_\lambda, \frac{L\left(x, \left(1+r(\lambda)\right)v\right) - L\big((1+r(\lambda))x,v\big)}{\lambda}\right\rangle_\Omega &+ u_\lambda(z) -\tilde{w}_\lambda(z) + \frac{c_\Omega}{\lambda} \nonumber\\
    &+(1+r(\lambda))^2\langle \tilde\mu_\lambda, w\rangle_\Omega    \geq 0.
\end{align}
By Lemma \ref{lem:u_gamma} there exists (up to subsequence) $\mu_0 \in \mathcal{M}(\Omega)$ such that $\tilde{\mu}_\lambda\rightharpoonup \mu_0$ in measure. Divide both sides by $\lambda$ and let $\lambda = \lambda_j \to  0$. We note that 
\begin{equation*}
    u_\lambda(z) + \frac{c_\Omega}{\lambda}  = (1+r(\lambda))^2 \left[\tilde{u}_\lambda\left(\frac{z}{1+r(\lambda)}\right) + \frac{c(0)}{\lambda}\right] + \left(\frac{1-(1+r(\lambda)^2)}{\lambda}\right)c_\Omega \rightarrow u(z) -2\gamma c_\Omega
\end{equation*}
as $\lambda_j\to 0$, thanks to \eqref{eq:vanishing-assume-two-limits}. We obtain from \eqref{eq:convergence-vanishing-abc4} that
\begin{equation*}
    \gamma\langle \mu_0, (-x,v)\cdot \nabla L(x,v) \rangle_\Omega + u(z) - w(z) - 2\gamma c_\Omega + \langle\mu_0, w\rangle_\Omega \geq 0.
\end{equation*}
and thus \eqref{eq:scale_down-2} follows. \medskip

\paragraph{\textbf{Step 3.}} We show the characterization 
\eqref{eq:maxM_gamma}. \medskip

\noindent
From \eqref{eq:scale_down-1} and \eqref{eq:scale_down-2} we conclude that $u\equiv w$. Let us denote the unique limit of $\tilde{u}_{\lambda}(x)+ \lambda ^{-1}c(0)$ by $\tilde{u}^\gamma$. We can get the relation between $\tilde{u}^\gamma$ and $u^\gamma$ by writing
\begin{align*}
    \tilde{u}_\lambda(x) + \frac{c_\Omega}{\lambda} &= \frac{1}{(1+r(\lambda))^2}\left(u_\lambda(x) +  \frac{c_\Omega}{ \lambda }\right)  \\
    &+ \frac{u_\lambda((1+r(\lambda))x) - u_\lambda(x)}{(1+r(\lambda))^2}  + \frac{c_\Omega}{ \lambda }\left(1-\frac{1}{(1+r(\lambda))^2}\right).
\end{align*}
Let $\lambda\to 0$ we obtain 
\begin{equation}\label{eq:relation-u-u-tilde}
    \tilde{u}^\gamma = u^\gamma +2\gamma c_\Omega.
\end{equation}
In \eqref{eq:scale_down-2} let $u = \tilde{u}^\gamma$ we obtain
\begin{equation*}
    \gamma\langle \mu_0, (-x,v)\cdot \nabla L(x,v)\rangle _\Omega  +\langle\mu_0,w\rangle_\Omega + u^\gamma(z) - w(z) \geq 0.
\end{equation*}
If $w\in \mathcal{E}^\gamma$ then by definition it implies $u^\gamma(z)\geq w(z)$. On the other hand, \eqref{eq:scale_down-1} reads $u^\gamma \in \mathcal{E}^\gamma$. Therefore $u^\gamma = \max \mathcal{E}^\gamma$.
\end{proof}

We omit the proof of the following result as it is similar to the first-order case in \cite[Corollary 1.3]{tu_vanishing_2021} once Theorem \ref{thm:general} is established. 

\begin{cor}\label{cor:concave_sln} 
The mapping $\gamma\mapsto u^\gamma(\cdot)$ is concave and decreasing. Precisely, if $\alpha\leq \beta$ then $u^\beta\leq u^\alpha$ and $(1-s)u^\alpha + s u^\beta \leq u^{(1-s)\alpha+s\beta}$ for $s\in (0,1)$.
\end{cor}

\section{The differentiability of the additive eigenvalue} \label{sec:c'(0)}
% So far we have known that $\lambda\mapsto c(\lambda)$ is bounded (Lemma \ref{lem:boundedness_c(lambda)}) and H\"older continuous (Corollary \ref{cor:Holder_c(lambda)}) on its domain. We show that it is one-sided differentiable everywhere and is also Lipschitz.

In this section, we prove Theorem \ref{thm:limit} and Corollary \ref{cor:diff_ae}. 
\subsection{The one-sided derivatives of the eigenvalue} 

\begin{proof}[Proof of Theorem \ref{thm:limit}] We will divide the proof of the limit in \eqref{eq:limit1} into three parts. Let us consider $\lambda 0$ first and prove \eqref{eq:limit1} as \eqref{eq:limit2} is similar. In steps 1 and 2, we establish inequalities for $\liminf$ and $\limsup$ of the limit \eqref{eq:limit1}. Then, we combine these inequalities together in step 3 to obtain the conclusion. \medskip 

\paragraph{\textbf{Step 1.}} We show that
\begin{equation}\label{eq:limit-p3-revise-1}
    \liminf_{\lambda\rightarrow 0^+}\left(\frac{c(\lambda)-c(0)}{\lambda}\right) \geq  \langle\mu, (-x,v)\cdot \nabla L(x,v)\rangle_\Omega \qquad\text{for all}\; \mu \in \mathcal{M}(\Omega).
\end{equation}
Let $w$ be a subsolution to
\begin{equation}\label{eq:cell-w_lambda}
    \begin{aligned}
        H(x,Dw(x)) - \varepsilon \Delta w(x) &\leq c(\lambda) &  &\text{in}\;\Omega_\lambda.
    \end{aligned}
\end{equation}
We scale \eqref{eq:cell-w_lambda} into an equation on $\Omega$ as follows:
\begin{equation}
    \begin{cases}
    \begin{aligned}
        &\tilde{w}(x) = \frac{w\left(1+\lambda\right)}{\left(1+\lambda\right)^2} &&\qquad\text{on}\;\overline{\Omega}.\\
        &H\Big(\left(1+\lambda\right)x, \left(1+\lambda\right) D\tilde{w}(x)\Big)  - \varepsilon \Delta \tilde{w}(x) \leq c(\lambda) &&\qquad\text{in}\;\Omega.
    \end{aligned}
    \end{cases}
\end{equation}
Then, using the strategy in \eqref{eq:form-1}, \eqref{eq:form-2} we obtain
\begin{equation*}
    \begin{cases}
    \begin{aligned}
        & H_\phi(x,D\tilde{w}(x)) - \varepsilon \Delta \tilde{w}(x) \leq 0  \qquad\text{in}\;\Omega\\
        & \phi(x,v) = L\left(\left(1+\lambda\right)x, \frac{v}{1+\lambda}\right) + c(\lambda) \in \Phi^+(\overline{\Omega}\times \R^n).
    \end{aligned}
    \end{cases}
\end{equation*}
Using Corollary \ref{cor:usage-measures} (ii) and the definition of $\mathcal{M}(\Omega)$ we obtain 
\begin{equation*}
    \begin{aligned}
         &\langle \mu, \phi\rangle_\Omega \geq 0 &&\qquad \mu\in \mathcal{M}(\Omega),\\
        -&\langle \mu, L\rangle_\Omega = c(0) &&\qquad \mu\in \mathcal{M}(\Omega). 
    \end{aligned}
\end{equation*}
Therefore
\begin{equation}\label{e:nice}
  \left\langle \mu, L\left(\left(1+\lambda\right)x, \frac{v}{1+\lambda}\right) - L(x,v)\right\rangle_\Omega + c(\lambda) - c(0) \geq 0, \qquad \forall\; \mu\in \mathcal{M}(\Omega).
\end{equation}
Thus if $\lambda>0$ we deduce that
\begin{equation}\label{eq:limit-p3}
    -\langle\mu, (-x,v)\cdot \nabla L(x,v)\rangle_\Omega + \liminf_{\lambda\rightarrow 0^+}\left(\frac{c(\lambda)-c(0)}{\lambda}\right) \geq 0 \qquad\forall\; \mu \in \mathcal{M}(\Omega).
\end{equation}
Therefore, we obtain our conclusion of the first step, equation \eqref{eq:limit-p3-revise-1}.
% \begin{equation}\label{eq:limit-p3-revise-1}
%     \liminf_{\lambda\rightarrow 0^+}\left(\frac{c(\lambda)-c(0)}{\lambda}\right) \geq  \langle\mu, (-x,v)\cdot \nabla L(x,v)\rangle_\Omega \qquad\text{for all}\; \mu \in \mathcal{M}(\Omega).
% \end{equation}
\medskip
\paragraph{\textbf{Step 2.}} We will show that there exists a measure $\sigma_0\in \mathcal{M}(\Omega)$ such that 
\begin{equation}\label{eq:limit-p3-revise-2}
    \limsup_{\lambda\rightarrow 0^+}\left(\frac{c(\lambda)-c_\Omega}{\lambda}\right) \leq  \langle\sigma_0, (-x,v)\cdot \nabla L(x,v)\rangle_\Omega.
\end{equation}
To see that, we choose $\lambda_j\rightarrow 0^+$ as a subsequence along which the $\limsup$ is attained:
\begin{equation}\label{eq:limit-p3-revise-3-limsup}
    \limsup_{\lambda\rightarrow 0^+}\left(\frac{c(\lambda)-c(0)}{\lambda}\right)  =\lim_{j\rightarrow \infty}\left(\frac{c(\lambda_j)-c(0)}{\lambda_j}\right).
\end{equation}
Let $u\in \mathrm{C}(\overline{\Omega})$ be such that
\begin{equation}\label{eq:step2-limsup-equation-Omega}
        H(x,Du(x)) - \varepsilon \Delta u(x) \leq c(0) \qquad  \text{in}\;\Omega.
\end{equation}
We scale \eqref{eq:step2-limsup-equation-Omega} into an equation on $\Omega_{\lambda}$ as follows: 
\begin{equation}\label{eq:eig_up2}
    \begin{cases}
    \begin{aligned}
        &\tilde{u}(x) = \left(1+\lambda\right)^2u\left(\frac{x}{1+\lambda}\right) &&\text{on}\;\overline{\Omega}_\lambda,\\ 
        &H\left(\frac{x}{1+\lambda}, \frac{1}{1+\lambda} D\tilde{u}(x)\right) - \varepsilon \Delta \tilde{u}(x) \leq c(0) \qquad  &&\text{in}\;\Omega_\lambda.
    \end{aligned}
    \end{cases}
\end{equation}
Then, using the strategy in \eqref{eq:form-1}, \eqref{eq:form-2} we obtain
\begin{equation*}
    \begin{cases}
    \begin{aligned}
        &H_\psi(x,D\tilde{w}(x)) - \varepsilon \,\Delta
    \tilde{w}(x)\leq 0 \qquad  \text{in}\;\Omega_\lambda\\
        & \psi(x,v) = L\left(\frac{x}{1+\lambda}, (1+\lambda)v\right) + c(0) \in \Phi^+(\overline{\Omega}_\lambda \times \R^n)
    \end{aligned}
    \end{cases}
\end{equation*}
Using Corollary \ref{cor:usage-measures} (ii) and definition of $\mathcal{M}(\Omega_\lambda)$ we have
\begin{equation*}
\begin{aligned}
    & \langle \sigma, \phi\rangle_{\Omega_\lambda } \geq 0 &&\qquad \sigma\in \mathcal{M}(\Omega_\lambda) \\
    -&\langle \sigma, L\rangle_{\Omega_\lambda} = c(\lambda) && \qquad
    \sigma\in \mathcal{M}(\Omega_\lambda).
\end{aligned}
\end{equation*}
Therefore 
% Thus $\psi(x,v)\in \mathcal{G}_{0, \Omega_\lambda}$. As $\nu_\lambda\in \mathcal{P}\cap \mathcal{G}'_{0,\Omega_\lambda}$ and $\langle \nu_\lambda, L\rangle_{\Omega_\lambda} = -c(\lambda)$, we obtain that
\begin{equation}\label{eq:step2-almost-final-form}
    \left\langle \sigma ,L\left(\frac{x}{1+\lambda}, (1+\lambda)v\right) - L\left(x,v\right)\right\rangle_{\Omega_\lambda} -c(\lambda)+c(0) \geq 0, \qquad \forall\; \sigma\in \mathcal{M}(\Omega_\lambda).
\end{equation}
Let $\sigma_{\lambda}$ be a measure in $\mathcal{M}(\Omega_{\lambda})$ and $\widetilde{\sigma}_\lambda$ be its scaling as in Definition \ref{defn:scaledown}. We write \eqref{eq:step2-almost-final-form} as 
\begin{equation}\label{e:nice3}
    \Big\langle \widetilde{\sigma}_\lambda, L\left(x,(1+\lambda)v\right) - L\left((1+\lambda)x,v\right)\Big\rangle_\Omega \geq c(\lambda)-c(0).
\end{equation}
If $\lambda> 0$, we have
\begin{equation}\label{e:nice2}
    \left\langle \widetilde{\sigma}_\lambda, \frac{L\left(x,(1+\lambda)v\right) - L\left((1+\lambda)x,v\right)}{\lambda}\right\rangle_\Omega \geq \frac{c(\lambda)-c(0)}{\lambda}.
\end{equation}
Along the sequence $\lambda_j$ that the limsup in \eqref{eq:limit-p3-revise-3-limsup} is attained, we can assume up to subsequence that $\widetilde{\sigma}_{\lambda_j}\rightharpoonup \sigma_0$ and $\sigma_0\in \mathcal{M}(\Omega)$. Let $\lambda_j\rightarrow 0^+$ in \eqref{e:nice2} we deduce that
\begin{equation}\label{eq:limit-p9}
    \left\langle \sigma_0, (-x,v)\cdot \nabla L(x,v)\right\rangle_\Omega \geq  \limsup_{\lambda\rightarrow 0^+} \left(\frac{c(\lambda) - c(0)}{\lambda}\right).
\end{equation}
Thus we obtain the conclusion \eqref{eq:limit-p3-revise-2} for our second steps. \medskip

\paragraph{\textbf{Step 3.}} We combine \eqref{eq:limit-p3-revise-1} (take $\mu = \sigma_0$) and \eqref{eq:limit-p3-revise-2} to obtain the conclusion \eqref{eq:limit1}, as:  
% In \eqref{eq:limit-p3-revise-1}, take $\mu = \sigma_0 \in \mathcal{M}(\Omega)$ and together with \eqref{eq:limit-p3-revise-2} we conclude that
\begin{equation*}
    \lim_{\lambda\rightarrow 0^+}\left(\frac{c(\lambda)-c(0)}{\lambda}\right)  =  \left\langle \sigma_0,(-x,v)\cdot \nabla L(x,v)\right\rangle_\Omega = \max_{\mu\in \mathcal{M}(\Omega)} \left\langle \mu,(-x,v)\cdot \nabla L(x,v)\right\rangle_\Omega.
\end{equation*}
Similarly, if  $\lambda\leq 0$ as $\lambda\rightarrow 0^+$ then we conclude \eqref{eq:limit2},
    \begin{equation*}
        \lim_{\lambda\rightarrow 0^+} \left(\frac{c(\lambda)-c(0)}{\lambda}\right) = \min_{\mu\in \mathcal{M}(\Omega)} \left\langle \mu,(-x,v)\cdot \nabla L(x,v)\right\rangle_\Omega.
    \end{equation*}

For an oscillating $\lambda$ such that neither $r^-(\lambda) = \min\{0,\lambda\}$ nor $r^+(\lambda) = \max\{0,\lambda\}$ is identical to zero as $\lambda\rightarrow 0^+$, by applying the previous results we obtain the final conclusion.
\end{proof}

\subsection{Addtional properties of the eigenvalue map}

\begin{lem}\label{lem:c} If $\Omega, \Omega'$ are two domains that satisfy \ref{itm:A1}. If $\Omega \subset \Omega'$ then $c_\Omega\leq c_{\Omega'}$.
\end{lem}
\begin{proof} If $c\in \mathbb{R}$ and $u$ is a subsolution for $H(x,Du) -\varepsilon \Delta u \leq c$ in $\Omega'$ then it is also a subsolution to the same equation in $\Omega$. Therefore, by \eqref{eq:c(0)}, we obtain the conclusion.
\end{proof}

\begin{lem}\label{rem: on L} Assume \ref{itm:A1}, \ref{itm:A2}. We have
\begin{equation}\label{eq:pos}
    \left\langle \mu, (-x,v)\cdot \nabla L(x,v)\right\rangle_\Omega \geq 0 \qquad \forall\;\mu\in \mathcal{M}(\Omega).
\end{equation}
% As a consequence, the one-sided derivatives with respect to $\lambda$ of $\lambda\mapsto c(\lambda)$ is nonnegative everywhere. And thus $\lambda\mapsto c(\lambda)$ is nondecreasing. 
\end{lem}

\begin{proof} The proof of \eqref{eq:pos} is a minor adaptation of \cite[Lemma 3.5]{tu_vanishing_2021}. 
% hence we omit it here. Now if we let $\theta = \lambda$ and $c(\theta) = c(\lambda) = c((1+\theta)\Omega)$, then the map $\theta\mapsto c(\theta)$ satisfies that 
% \begin{equation*}
%     c'_{\pm}(\theta) \geq 0
% \end{equation*}
% for every $\theta$. Therefore $\theta \mapsto c(\theta)$ is nondecreasing, i.e., $\lambda\mapsto c(\lambda)$ is nondecreasing.
\end{proof}

From the proof of Theorem \ref{thm:limit}, we see that equations \eqref{e:nice} and \eqref{e:nice2} imply that $\lambda\mapsto c(\lambda)$ is indeed Lipschitz, which is stronger than Corollary \ref{cor:Holder_c(lambda)}. We verify that in the following Lemma.

\begin{lem}\label{prop:bound<mu,L>} Assume \ref{itm:A1}, \ref{itm:A2}. There exists $C = C(p,\mathrm{diam}(U), \min_U f,\kappa)$ where $\kappa$ is defined in \ref{itm:A1} such that
\begin{equation*}
    \langle \mu, |v|^q\rangle_{\Omega_\lambda} \leq C \qquad\text{for all}\;  \mu \in \mathcal{M}(\Omega_\lambda).
\end{equation*}
Consequently, there exists a constant $C$ such that 
\begin{equation}\label{eq:Lipschitz-clambda}
    |c(\lambda) - c(0)|\leq C|\lambda|.
\end{equation}
\end{lem}

\begin{proof} For $\mu\in \mathcal{M}(\Omega_\lambda)$ we have
\begin{align*}
    -c(\lambda) = \langle \mu, L\rangle_{\Omega_\lambda} = \big\langle \mu, C_q|v|^q + f(x)\big\rangle_{\Omega_\lambda} = C_q\langle \mu, |v|^q\rangle_{\Omega_\lambda} + \langle \mu, f\rangle_{\Omega_\lambda}.
\end{align*}
Therefore
\begin{equation*}
    \left\langle \mu, |v|^q\right\rangle_{\Omega_\lambda}  = -\frac{c(\lambda) + \langle\mu, f\rangle_{\Omega_\lambda}}{C_q} 
    \leq   \frac{1}{C_q}\left(\max_{\overline{U}} |f| + C\right)
\end{equation*}
where $C$ is the constant defined in Lemma \ref{lem:bound_delta_u_delta_eig} and $U$ is defined in \ref{itm:A1}. To prove \eqref{eq:Lipschitz-clambda}, using $L(x,v) = C_p|v|^q + f(x)$ we observe that for $s\in (-1,1)$ then $(1+s)^q - 1 \leq q^2$ for $s\in (-1,1)$ and $q\in (1,2]$, therefore
    \begin{align*}
        &|L\big(x, (1+s)v\big) - L\big((1+s)x, v\big)| \\
        &\qquad\qquad \leq C_p\left((1+s)^q-1\right)|v|^q + \left|f(x) - f((1+s)x)\right| \leq \left(C_pq^2|v|^q + C\right)|s|
    \end{align*}
    where $C$ depends on the Lipschitz constant of $f$ on $U$ and $\mathrm{diam}(U)$. Let us consider $\lambda > 0$ first, then from \eqref{e:nice2} and the fact that $c_\Omega\leq c(\lambda)$ (Lemma \ref{lem:c}) we have 
    \begin{equation*}
    \begin{aligned}
        0 \leq c(\lambda)-c(0) &\leq \left\langle \tilde{\nu}_\lambda, L\left(x,(1+\lambda)v\right) - L\left((1+\lambda)x,v\right)\right\rangle_\Omega \\
        &\leq \left(C_pq^2\left\langle \tilde{\nu}_\lambda, |v|^q\right\rangle_\Omega + C\right)|\lambda| \\
        & = \left(C_pq^2\left\langle \nu_\lambda, |v|^q\right\rangle_{\Omega_\lambda} + C\right)|\lambda|  \leq C(C_pq^2+1) |\lambda|.
    \end{aligned}
    \end{equation*}
    The case where $\lambda \leq 0$ is similar. We conclude that $\lambda\mapsto c(\lambda)$ is Lipschitz.
\end{proof}

We finish this section by proving Corollary \ref{cor:diff_ae}.

\begin{proof}[Proof of Corollary \ref{cor:diff_ae}] Under \ref{itm:A3}, the fact that $\lambda \mapsto c(\lambda)$ is Lispchitz, increasing and $c'_{\pm}(\lambda)$ exists are consequences of Lemmas \ref{lem:c}, \ref{rem: on L} and \ref{prop:bound<mu,L>}. Parts (ii) is similar to the first-order case in \cite[Theorem 4.4]{tu_vanishing_2021}. Thus, we omit the proof of this fact. Part (iii) is a result from \cite[Theorem 4.2, Chapter 4]{bruckner_differentiation_1978} or \cite[ Theorem 17.9]{hewitt_real_1965}, where if $c'_{\pm}(\lambda)$ exists everywhere then the set where they are different is at most countable.
    % We present the proof for this in Lemma \ref{rem: on L}, which we can follow \cite[Lemma 3.5]{tu_vanishing_2021} with a minor adaptation, which we omit. In fact, one can show further that, if $\Omega\subset \Omega'$ then $c_{\Omega} \leq c_{\Omega'}$ where $c_\Omega$ and $c_{\Omega'}$ are ergodic constants of $H$ with respect to $\Omega, \Omega'$ respectively. We refer to \cite{barles_large_2010} for this fact. 
\end{proof} 

Together with Theorem \ref{thm:Holder_est_nested_domain}, we use the fact that $\lambda\mapsto c(\lambda)$ is Lipschitz to improve the error estimate in the nested domain setting but with some trade-off by a discount factor. 

\begin{cor}[Error estimate on nested domain] Assume \ref{itm:A1} and \ref{itm:A2}, $\theta >0$ and $0<\delta < 1$. Let $u \in \mathrm{C}((1+\theta)\overline{\Omega})$ and $v\in \mathrm{C}(\overline{\Omega})$ be solutions to $(\delta, (1+\theta)\Omega)$ and $(\delta, \Omega)$, respectively. Then, it holds that
\begin{equation*}
    0 \leq \delta  v(x) -  \delta u(x) \leq C\min \big\lbrace \delta+\theta, \theta^\alpha\big\rbrace, \qquad \text{for all}\; x\in \overline{\Omega}.
\end{equation*}
where $\alpha = (p-2)/(p-1)$ and $C$ depends on $\mathrm{diam}(U)$, the Lipschitz constant of $f$ on $U$ and $p$.
\end{cor}

\section{The ergodic problem and the differentiability of the additive eigenvalue}
\label{sec:ergodic problem and the differentiability} 
In this section, we prove Theorem~\ref{cor:back_forth}. From Theorem \ref{thm:limit}, when talking about $c'_{\pm}(0)$ we think about the derivatives of the map $\theta \mapsto c_{(1+\theta)\Omega}$ for $\theta \in (-\varepsilon_0,\varepsilon_0)$, with the formulas provided by \eqref{eq:limit1} and \eqref{eq:limit2}.

We first recall the definition of $\mathcal{C}(\gamma) = u^\gamma(\cdot) - u^0(\cdot)$ as a function from $\R\to \R$, from Remark \ref{rem:C(gamma)}.
We collect some definitions on the minimizing measures following \eqref{eq:representation_main-lambda} in Theorem \ref{thm:repdelta>0_main}.

\begin{defn}[Discount measures]\label{defn:M(lambda-z-Omega)}
    For $\lambda > 0$ and $z\in \overline{\Omega}$, we define
    \begin{align}
        \mathcal{M}(\lambda, z, \Omega) &= \left\lbrace \mu \in \mathcal{P}\cap \mathcal{G}'_{z,\lambda,\Omega}: \lambda u_\lambda(z) = \langle\mu, L\rangle_\Omega \right\rbrace, \label{eq:M(lambda-z-Omega)}\\
        \mathcal{M}(\lambda,\Omega) &= \bigcup_{z\in \overline{\Omega}} \mathcal{M}(\lambda, z,\Omega) \label{eq:M(lambda-Omega)}.
    \end{align}
    We say measures in $\mathcal{M}(\lambda,\Omega)$ are \emph{discount measures}.
\end{defn}

It is clear that weak limits of measures in $\mathcal{M}(\lambda,\Omega)$ as $\delta\to 0^+$ are members of $\mathcal{M}(\Omega)$. Regarding Lemma \ref{lem:u_gamma}, we define the following sets of measures.

\begin{defn}\label{def:solve_measures} For $\gamma\in \R$, $z\in \overline{\Omega}$, we define 
\begin{equation}\label{eq:r-lambda-linear}
    r_\gamma(\lambda) = \gamma\lambda, \qquad \Omega_\lambda = (1+\gamma\lambda)\Omega,  \qquad \lambda > 0
\end{equation}
and 
\begin{align}
    \mathcal{U}_\gamma(\Omega,z) &= \big\lbrace \mu\in \mathcal{M}(\Omega) :\;\exists\,\mu_{\lambda}\in  \mathcal{M}( \lambda , z, \Omega_\lambda)\;\text{such that}\; \tilde{\mu}_\lambda\rightharpoonup \mu \;\text{along a subsequence}\big\rbrace, \label{eq:U-gamma-Omega-z}\\
    \mathcal{U}_\gamma(\Omega) &= \big\lbrace \mu\in \mathcal{M}(\Omega) :\;\exists\,\mu_{\lambda}\in  \mathcal{M}( \lambda , \Omega_\lambda)\;\text{such that}\; \tilde{\mu}_\lambda\rightharpoonup \mu \;\text{along a subsequence}\big\rbrace,\\
    \mathcal{V}_\gamma(\Omega) &= \big\lbrace \nu\in \mathcal{M}(\Omega) :\;\exists\,\nu_{\lambda}\in  \mathcal{M}(\Omega_\lambda)\;\text{such that}\; \tilde{\nu}_\lambda\rightharpoonup \nu \;\text{along a subsequence}\big\rbrace,
\end{align}
where $\tilde{\mu}_\lambda, \tilde{\nu}_\lambda$ are the scaling measure defined on $\overline{\Omega}\times\R^n$ as in Definition \ref{defn:scaledown}.
\end{defn} 
We note that this choice of $r$ does not affect $u^\gamma$ as indicated by Theorem \ref{thm:general}. It is clear that $\bigcup_{z\in \overline{\Omega}} \mathcal{U}_\gamma(\Omega,z)\subset \mathcal{U}_\gamma(\Omega)$, and $\mathcal{U}_\gamma(\Omega), \mathcal{V}_\gamma(\Omega)$ are independent of the vertex $z\in \overline{\Omega}$ by the way we defined $\mathcal{M}( \lambda , \Omega_\lambda)$. Using \cite[Lemma 2.1 and Lemma 2.2]{ishii_vanishing_2017}, we have $\emptyset \neq \mathcal{V}_\gamma(\Omega),\mathcal{U}_\gamma(\Omega) \subset \mathcal{M}(\Omega)$.
%As $c'_{\pm}(0)$ is independent of the choice of $r(\lambda)$ (Theorem \ref{thm:limit}), and furthermore $u^\gamma$ in Theorem \ref{thm:general} does not depend of $r(\lambda)$ but only the limit $\gamma = \lim_{\lambda\to 0} r(\lambda)/\lambda$. 
If $\gamma = 0$ then we work on a fixed domain $\Omega$, and 
\begin{equation*}
    \mathcal{U}_0(\Omega) = \left\lbrace \mu \in \mathcal{M}(\Omega): \mu\;\text{is a weak limit of a sequence of measures in}\;\mathcal{M}(\lambda,\Omega)\right\rbrace.
\end{equation*}
Here $\mathcal{M}(\lambda,\Omega)$ is defined Definition \ref{defn:M(lambda-z-Omega)}. In other words, measures in $\mathcal{U}_0(\Omega)$ are weak limits of discount measures on $\Omega$, while measures in $\mathcal{U}_\gamma(\Omega)$ are weak limits of discount measures on $\Omega_\lambda = (1+\gamma \lambda)\Omega$.

% We define . We prove  about the equivalence between the differentiability of $\lambda\mapsto c(\lambda)$ and $\gamma\mapsto \mathcal{C}(\gamma)$. We first state some properties of measures in $\mathcal{U}_0(\Omega)$ (Definition \ref{def:solve_measures}) as follows.

\begin{cor}\label{cor:properties_U_gamma(z)} Assume \ref{itm:A1}. We have
\begin{equation*}
    \gamma \langle \mu_\gamma, (-x,v)\cdot \nabla L(x,v)\rangle_\Omega + \left\langle \mu, u^\gamma\right\rangle_\Omega = 0 \qquad\text{for any}\;\mu\in \mathcal{U}_\gamma(\Omega).
\end{equation*}
\end{cor}

\begin{proof} Take $\mu\in \mathcal{U}_\gamma(\Omega)$, we can find a sequence of points $z_\lambda\in \Omega_\lambda$ and $\mu_\lambda \in\mathcal{M}( \lambda , z_\lambda, \Omega_\lambda)$ such that $ \lambda u_\lambda(z_\lambda) = \langle \mu_\lambda, L \rangle_{\Omega_\lambda}$ and $\tilde{\mu}_\lambda \rightharpoonup \mu$ weakly in measures where $\tilde{\mu}_\lambda$ is the measure obtained from $\mu_\lambda$ defined as in Definition \ref{defn:scaledown}. Going through Step 2 as in the proof of Theorem \ref{thm:limit} with $w = \tilde{u}^\gamma$ as in \eqref{eq:convergence-vanishing-abc4}, we obtain 
\begin{equation*}
    \gamma\big\langle \mu, (-x,v)\cdot \nabla L(x,v)\big\rangle_\Omega + \langle\mu, u^\gamma\rangle_\Omega  \geq 0
\end{equation*}
and thus, the conclusion follows, as the other side of the inequality is given in \eqref{eq:scale_down-1}. Note that we use \eqref{eq:relation-u-u-tilde} to connect $\tilde{u}^\gamma$ to $u^\gamma$.
\end{proof}

\begin{lem}\label{lem:sigma0} Assume \ref{itm:A1}.
\begin{itemize}
    \item[$\mathrm{(i)}$] We have $\langle \sigma, u^0\rangle_\Omega =0$ for all $\sigma \in \mathcal{U}_0(\Omega)$. Furthermore
    \begin{equation*}
        \mathcal{C}(\gamma) = u^\gamma(z) - u^0(z) = \langle \sigma, u^\gamma\rangle_\Omega \qquad\text{for all}\;z\in \overline{\Omega}, \sigma\in \mathcal{U}_0(\Omega).
    \end{equation*}
    As a consequence, we have
    \begin{equation*}
        \gamma\big\langle \sigma,(-x,v)\cdot \nabla L(x,v) \big\rangle_\Omega + \mathcal{C}(\gamma)  \leq 0 \qquad\text{for all}\;\sigma\in \mathcal{U}_0(\Omega).
    \end{equation*}
    \item[$\mathrm{(ii)}$] We have
    \begin{equation}\label{eq:>=0-for-U-gamma}
        \gamma\big\langle \mu,(-x,v)\cdot \nabla L(x,v) \big\rangle_\Omega + \mathcal{C}(\gamma) \geq 0 \qquad\text{for all}\;\mu\in \mathcal{U}_\gamma(\Omega).
    \end{equation}
    
    \item[$\mathrm{(iii)}$] The map $\mathcal{C}:\gamma\mapsto \mathcal{C}(\gamma)$ is decreasing, concave with $\mathcal{C}(0) = 0$.

    \item[$\mathrm{(iv)}$] We have
    \begin{align*}
        \langle \sigma, u^\gamma\rangle_\Omega + \gamma c'_+(0) \leq 0 \qquad\text{for all}\;\sigma \in \mathcal{V}_{\gamma}(\Omega), \gamma > 0,\\
        \langle \sigma, u^\gamma\rangle_\Omega + \gamma c'_-(0) \leq 0 \qquad\text{for all}\;\sigma \in \mathcal{V}_{\gamma}(\Omega), \gamma < 0.
    \end{align*}
\end{itemize}
\end{lem}

\begin{proof} We recall that $\mathcal{U}_0,\mathcal{U}_\gamma \subset \mathcal{M}(\Omega)$. 
\begin{itemize}
    \item[(i)] By Theorem \ref{thm:general} we have $\langle \sigma, u^0\rangle \leq 0$ for all $\sigma \in \mathcal{M}(\Omega)$. Take $\sigma\in \mathcal{U}_0(\Omega)$, there exist $z_\lambda\in \overline{\Omega}$ and $\sigma_\lambda\in \mathcal{M}( \lambda , z_\lambda, \Omega)$ such that (up to a subsequence) $\sigma_\lambda\rightharpoonup \sigma$ weakly in measures. Let $w$ be a solution to $H(x,Dw(x)) - \varepsilon\Delta w(x) \leq c_\Omega$ in $\Omega$. Using the strategy in \eqref{eq:form-1}, \eqref{eq:form-2} we write 
    \begin{equation*}
    \begin{cases}
        \lambda w + H_\phi(x,Dw) - \varepsilon \Delta w \leq 0 \qquad\text{in}\;\Omega, \\
        \phi(x,v) = L(x,v) +  \lambda w(x) + c_\Omega \in \Phi^+(\overline{\Omega}\times\R^n).
    \end{cases}
    \end{equation*}
    Using Corollary \ref{cor:usage-measures} (i) we have
    \begin{equation*}
            \langle \sigma_\lambda, \phi - \lambda w(z)\rangle_{\Omega} \geq 0.
    \end{equation*}
    Using the definition of $\sigma_\lambda$ we have $ \langle \sigma_\lambda, L\rangle_\Omega =  \lambda v_\lambda(z_\lambda)$ where $v_\lambda$ solves \eqref{eq:main}. Therefore
    \begin{equation*}
         \lambda v_\lambda(z_\lambda) + c_\Omega +  \lambda \langle \sigma_\lambda, w\rangle_\Omega -  \lambda w(z_\lambda) \geq 0.
    \end{equation*}
    Divide both sides by $ \lambda $ we obtain
    \begin{equation*}
        \left(v_\lambda(z_\lambda) +  \frac{c_\Omega}{\lambda}\right) + \langle \sigma_\lambda, w\rangle_\Omega - w(z_\lambda) \geq 0.
    \end{equation*}
    We can assume $z_\lambda\to z_0$ for some $z_0\in \overline{\Omega}$. Take $\lambda\to 0^+$, thanks to Theorem \ref{thm:general} we have
    \begin{equation}\label{eq:lem-U0-Ugamma-connect-i}
        u^0(z_0) + \langle \sigma, w\rangle_\Omega \geq w(z_0)
    \end{equation}

    \begin{itemize}
        \item[$\circ$] Let $w = u^0$ in \eqref{eq:lem-U0-Ugamma-connect-i} then $\langle \sigma, u^0\rangle \geq 0$, thus 
        \begin{equation}\label{eq:<mu0,u0>-zero}
            \langle \sigma, u^0\rangle = 0 \qquad\text{for}\;\sigma \in \mathcal{U}_0(\Omega),
        \end{equation}
        since $\langle \mu, u^0\rangle\leq 0$ for all $\mu\in \mathcal{M}(\Omega)$. 

        \item[$\circ$] On the other hand, using \eqref{eq:<mu0,u0>-zero} we have
        \begin{align}\label{eq:<mu0,u-gamma>-C-gamma}
            \langle \sigma, u^\gamma\rangle_\Omega = \langle \sigma, u^\gamma\rangle_\Omega - \langle \sigma, u^0\rangle_\Omega = \langle \sigma, u^\gamma - u^0\rangle_\Omega = \langle \sigma, \mathcal{C}(\gamma)\rangle_\Omega = \mathcal{C}(\gamma).
        \end{align}
        
        % \item[$\circ$] Let $w = u^\gamma$ in \eqref{eq:lem-U0-Ugamma-connect-i} then
        % \begin{equation*}
        %     u^0(z_0) + \langle \sigma, u^\gamma\rangle_\Omega \geq u^\gamma(z_0)  
        %     \qquad\Longrightarrow\qquad \langle \sigma, u^\gamma\rangle_\Omega \geq u^\gamma(z_0) - u^0(z_0).
        % \end{equation*}
        % Here we use Remark \ref{rem:C(gamma)} and Definition \ref{defn:C(gamma)} that $u^\gamma(x) - u^0(x) = \mathrm{C}(\gamma)$ for all $x\in \overline{\Omega}$, we deduce that 
        % \begin{equation}\label{eq:lem-U0-Ugamma-connect-i-formula-C-gamma}
        %     \langle \sigma, u^\gamma\rangle_\Omega  \geq  u^\gamma(z_0) - u^0(z_0) =\mathcal{C}(\gamma)  \qquad\text{for all}\;\sigma \in \mathcal{U}_0(\Omega).
        % \end{equation}
    \end{itemize}
    % Now we have
    
    % \begin{align*}
    %     \mathcal{C}(\gamma) 
    %         = \langle \sigma, \mathcal{C}(\gamma)\rangle_\Omega 
    %         = \langle \sigma, u^\gamma - u^0\rangle_\Omega = \langle \sigma\rangle_\Omega = 
    % \end{align*}
    % thanks to 
    \noindent
    In equation \eqref{eq:maxM_gamma} of Theorem \ref{thm:general} (ii), let $w=u^\gamma$ and use \eqref{eq:<mu0,u-gamma>-C-gamma} we obtain
    \begin{equation*}
        \gamma\big\langle \sigma,(-x,v)\cdot \nabla L(x,v) \big\rangle_\Omega + \mathcal{C}(\gamma)  = \gamma\big\langle \sigma,(-x,v)\cdot \nabla L(x,v) \big\rangle_\Omega + \langle \sigma, u^\gamma\rangle_\Omega \leq 0.
    \end{equation*}
    \item[(ii)] Since $\mathcal{C}(\gamma) = u^\gamma(\cdot) - u^0(\cdot)$ is a constant, we have 
    \begin{equation*}
        \mathcal{C}(\gamma) = \langle \mu, \mathcal{C}(\gamma)\rangle_\Omega = \langle \mu, u^\gamma - u^0\rangle_\Omega \qquad\Longrightarrow\qquad \langle \mu, u^\gamma\rangle_\Omega =\mathcal{C}(\gamma) + \langle \mu, u^0\rangle_\Omega .
    \end{equation*}
    Using this in Corollary \ref{cor:properties_U_gamma(z)} we have
    \begin{equation*}
    \begin{aligned}
        0 &= \gamma\big\langle \mu, (-x,v)\cdot \nabla L(x,v) \big\rangle_\Omega + \langle \mu, u^\gamma\rangle_\Omega \\
        &= \gamma\big\langle \mu, (-x,v)\cdot \nabla L(x,v) \big\rangle_\Omega + \mathcal{C}(\gamma) + \langle\mu, u^0\rangle_\Omega \qquad\text{for all}\;\mu\in \mathcal{U}_\gamma(\Omega).
    \end{aligned}
    \end{equation*}
    Since $\langle \mu,u^0 \rangle \leq 0$ for all $\mu \in \mathcal{M}(\Omega)$ (Theorem \ref{thm:general}) and $\mathcal{U}_\gamma\subset \mathcal{M}(\Omega)$, we obtain the desired conclusion.

    \item[(iii)] From part (i) we have $\langle \mu, u^\gamma\rangle = \mathcal{C}(\gamma)$ for all $\mu\in \mathcal{U}_0(\Omega)$. By Corollary \ref{cor:concave_sln} $\gamma\mapsto u^\gamma(\cdot)$ is decreasing and concave. Take $\sigma\in \mathcal{U}_0(\Omega)$, then $\langle \sigma, u^0\rangle_\Omega = 0$ from part (i). If $\alpha \leq \beta$ as real numbers then $u^\alpha - u^\beta  \geq 0$, thus $\langle \sigma, u^\alpha - u^\beta  \rangle_\Omega\geq 0$, hence
    \begin{equation*}
           \langle \sigma, u^\alpha - u^0\rangle_\Omega - \langle \sigma, u^\beta - u^0\rangle_\Omega \geq 0 \quad\Longrightarrow\quad \mathcal{C}(\alpha)\geq \mathcal{C}(\beta).
    \end{equation*}
    Therefore $\gamma\mapsto \mathcal{C}(\gamma)$ is decreasing. On the other hand, for $s\in (0,1)$ and $\alpha, \beta \in \R$ then $s u^\alpha + (1-s)u^\beta \leq u^{s\alpha+(1-s)\beta}$, hence the concavity of $\gamma\mapsto \mathcal{C}(\gamma)$ follows from part (i). 
    \item[(iv)] 
Let $u_\lambda\in \mathrm{C}(\overline{\Omega}_\lambda)$ be the solution to \eqref{eq:S_lambda} then $\lambda u_\lambda + H(x,Du_\lambda) - \varepsilon\Delta u_\lambda \leq 0$ in $\Omega_\lambda$. Using the strategy in \eqref{eq:form-1}, \eqref{eq:form-2} we write 
    \begin{equation*}
    \begin{cases}
        H_\phi(x,Du_\lambda(x)) - \varepsilon\Delta u_\lambda(x)\leq 0 \qquad\text{in}\;\Omega_\lambda, \\
        \phi(x,v) = L(x,v) -  \lambda u_\lambda(x)\in \Phi^+(\overline{\Omega}\times\R^n).
    \end{cases}
    \end{equation*}
    Using Corollary \ref{cor:usage-measures} (ii) we have $\langle \sigma_\lambda, \phi\rangle_{\Omega_\lambda} \geq 0$ and $\langle \sigma_\lambda, L\rangle_{\Omega_\lambda} = -c_{\Omega_\lambda}$ if $\sigma_\lambda \in \mathcal{M}(\Omega_\lambda)$, therefore
    \begin{equation*}
         \lambda \langle \sigma_\lambda,-c_{\Omega_\lambda} - \lambda u_\lambda\rangle_{\Omega_\lambda} = \left\langle \sigma_\lambda,  L(x,v) -  \lambda u_\lambda(x)\right\rangle_{\Omega_\lambda} \geq 0 
    \qquad\text{for all}\;\sigma_\lambda\in \mathcal{M}(\Omega_\lambda).
    \end{equation*}
    In other words, we have 
    \begin{equation*}
        \left\langle \sigma_\lambda , u_\lambda(x) + \frac{c_{\Omega_\lambda}}{ \lambda }\right\rangle_{\Omega_\lambda} \leq 0 \qquad\text{for all}\;\sigma_\lambda\in \mathcal{M}(\Omega_\lambda).
    \end{equation*}
    Let $\widetilde{\sigma}_\lambda$ be the scaling measure obtained from $\sigma_\lambda$ as in Definition \ref{defn:scaledown}. We have
    \begin{equation}\label{eq:V-lambda-measure}
        \left\langle \widetilde{\sigma}_\lambda, u_\lambda((1+r(\lambda))x) + \frac{c_\Omega}{ \lambda }\right\rangle_{\Omega} + \frac{c_{\Omega_\lambda} - c_\Omega}{ \lambda } \leq 0\qquad\text{for all}\;\sigma_\lambda\in \mathcal{M}(\Omega_\lambda).
    \end{equation}
    If $\sigma \in \mathcal{V}_\gamma(\Omega)$, we can find a sequence $\sigma_\lambda\in \mathcal{M}(\Omega_\lambda)$ such that $\sigma_\lambda\rightharpoonup \sigma$ in measures (thanks to Lemma \ref{lem:u_gamma}). Let $\sigma = \sigma_\lambda$ in \eqref{eq:V-lambda-measure}, then as $\lambda\to 0^+$ we obtain the conclusion. 
    % \begin{itemize}
    %     \item[$\circ$] If $\gamma>0$ then $\langle \sigma, u^\gamma\rangle_\Omega + \gamma c'_+(0) \leq 0$.
    %     \item[$\circ$] If $\gamma<0$ then $\langle \sigma, u^\gamma\rangle_\Omega + \gamma c'_-(0) \leq 0$.
    % \end{itemize}
\end{itemize}
\end{proof}

\begin{proof}[Proof of Theorem~\ref{cor:back_forth}] Using the concavity of $\gamma\mapsto \mathcal{C}(\gamma)$, the one-sided derivatives $\mathcal{C}'_{\pm}(\gamma)$ exists everywhere. From Lemma \ref{lem:sigma0} we have
\begin{align}
   \langle \mu_\gamma, (-x,v)\cdot \nabla L(x,v)\rangle_\Omega &\geq -\frac{\mathcal{C}(\gamma)}{\gamma} \qquad\forall\;\mu \in \mathcal{U}_\gamma(\Omega), \gamma>0 \label{eq:C-gamma-1}\\
   \langle \mu_\gamma, (-x,v)\cdot \nabla L(x,v)\rangle_\Omega &\leq -\frac{\mathcal{C}(\gamma)}{\gamma} \qquad\forall\;\mu \in \mathcal{U}_\gamma(\Omega), \gamma<0. \label{eq:C-gamma-2}
\end{align}
Take a subsquence $\mu_\gamma \rightharpoonup \mu_+$ for some $\mu_+\in \mathcal{M}(\Omega)$ as $\gamma\to 0^+$ we obtain
\begin{equation*}
     \langle \mu_+, (-x,v)\cdot \nabla L(x,v)\rangle_\Omega \geq -\mathcal{C}'_+(0).
\end{equation*}
Similarly, take a subsquence $\mu_\gamma \rightharpoonup \mu_-$ for some $\mu_-\in \mathcal{M}(\Omega)$ as $\gamma\to 0^-$ we obtain
\begin{equation*}
     \langle \mu_-, (-x,v)\cdot \nabla L(x,v)\rangle_\Omega \leq -\mathcal{C}'_-(0).
\end{equation*}
As a concave function, we have $\mathcal{C}'_+(0) \leq \mathcal{C}'_-(0)$, therefore
\begin{equation*}
    \langle \mu_-, (-x,v)\cdot \nabla L(x,v)\rangle_\Omega \leq -\mathcal{C}'_-(0) \leq  -\mathcal{C}'_+(0) \leq  \langle \mu_+, (-x,v)\cdot \nabla L(x,v)\rangle_\Omega. 
\end{equation*}
If $c'(0)$ exists, then since
\begin{equation}\label{eq:c-derivative-constant}
    \langle \mu, (-x,v)\cdot \nabla L(x,v)\rangle_\Omega = c'(0) \qquad\text{for all}\; \mu\in \mathcal{M}(0),
\end{equation}
we deduce that $-\mathcal{C}'(0) = c'(0)$. Furthermore, from \eqref{eq:C-gamma-1}, \eqref{eq:C-gamma-2} and \eqref{eq:c-derivative-constant} we obtain
\begin{equation*}
    -\frac{\mathcal{C}(\gamma)}{\gamma} = c'(0) \qquad\text{for all}\;\gamma.
\end{equation*}
Therefore $\mathcal{C}(\gamma) = -\gamma c'(0)$.
\end{proof}

\section{Remarks on the differentiability of the additive eigenvalue map}
\label{sec:rem}
Let $r(\lambda) = \lambda$ for $\lambda\in (-\varepsilon_0, \varepsilon_0)$. 
While Corollary \ref{cor:diff_ae} guarantees that $\lambda\mapsto c(\lambda)$ is differentiable except for a countable set, the following question remains open.

\begin{quest} 
Can we show that $\lambda\mapsto c(\lambda)$ is indeed differentiable everywhere?
\end{quest}

In this section, we demonstrate the smoothness of $\lambda\mapsto c(\lambda)$ for constant $f$, discuss its semi-convexity for semi-concave data, and briefly outline the smoothness of $\lambda\mapsto c(\lambda)$ for $p=2$, with an explicit formula for $c'(0)$ using the Hopf-Cole transformation. Investigation of the full regime $1<p\leq 2$ is planned for future work.

% In this section, we first show that $\lambda\mapsto c(\lambda)$ is smooth when $f$ is a constant. We then discuss the semi convexity of this map when the data is semi concave. Finally, we briefly give a formal argument on the smoothness of the map $\lambda\mapsto c(\lambda)$ when $p=2$ with an explicit formula for $c'(0)$, using linear elliptic theory via the Hopf-Cole transformation. The full regime $1<p\leq 2$ will be investigated in the future. 

\subsection{Smoothness with constant data}
% We present here the proof for Proposition~\ref{lem:example}, that $\lambda\mapsto c(\lambda)$ is differentiable everywhere if $f\equiv \mathrm{const}$.

\begin{proof}[Proof of Theorem~\ref{lem:example}] Without loss of generality we can assume $f\equiv 0$, then $L(x,v) = C_p|v|^q$ where $C_p = p^{-1/q}(p-1)$ and $p^{-1} + q^{-1} = 1$. By definition of $\mathcal{M}(\Omega)$ we have 
\begin{equation*}
    -c(0) =  \langle \mu, L\rangle_\Omega = C_p \left\langle \mu, |v|^q\right\rangle_\Omega \qquad\text{for any}\;\mu\in \mathcal{M}(\Omega).
\end{equation*}
Since $(-x,v)\cdot \nabla L(x,v) = qC_p|v|^{q}$, we have
\begin{equation*}
    \big\langle \mu,  (-x,v)\cdot \nabla L(x,v)\big\rangle_\Omega = qC_p \left\langle \mu, |v|^q\right\rangle_\Omega = -qc(0) \qquad\text{for all}\;\mu \in \mathcal{M}(\Omega).
\end{equation*}
In view of Theorem \ref{thm:limit} we conclude that $c'(0)$ exists and $c'(0) = -p(p-1)^{-1}c(0)$. As the argument can be done for any $\Omega_\lambda$ we obtain $\lambda\mapsto c(\lambda)$ is differentiable everywhere and
\begin{equation}\label{eq:f=0_case}
    c'(\lambda) = -p(p-1)^{-1}c(\lambda) = -qc(\lambda).
\end{equation}
% From Corollary \ref{cor:diff_ae}, we obtain that $\lambda\mapsto c(\lambda)$ is continuously differentiable. Equation \eqref{eq:f=0_case} also implies that $\lambda\mapsto c(\lambda)$ belongs to $\mathrm{C}^\infty$ with $c^{(k)}(\lambda) = (-q)^kc(\lambda)$ for $k\in \mathbb{N}$.
From Corollary \ref{cor:diff_ae}, we see that $\lambda\mapsto c(\lambda)$ is continuously differentiable. Equation \eqref{eq:f=0_case} implies that $\lambda\mapsto c(\lambda)$ is $\mathrm{C}^\infty$, with $c^{(k)}(\lambda) = (-q)^kc(\lambda)$ for $k\in \mathbb{N}$.
\end{proof}

\subsection{Semiconvexity with semiconcave data}

% We see that the higher regularity of $\lambda\mapsto c(\lambda)$ is a natural question. 
Another equivalent definition of a semiconvex function is as follows: there is some $\tau > 0$ such that
\begin{equation*}
    u\big(sx+(1-s)y\big) \leq su(x) + (1-s)u(y) + s(1-s)\frac{|x-y|^2}{2\tau}, \qquad x,y\in I, s\in [0,1].
\end{equation*}

\begin{proof}[Proof of Theorem \ref{prop:semiconvex}] As usual we only need to consider $r(\lambda) = \lambda$. Let $\alpha<\eta<\beta$ and the corresponding $\left(w_\alpha,c_\alpha\right),\left(w_\eta,c_\eta\right),\left(w_\beta,c_\beta\right)$ be solutions and eigenvalues to the ergodic problem where $\eta = s\alpha+(1-s)\beta$ for some $s\in [0,1]$. Let us define for $\lambda\in \R$ the function
\begin{equation*}
    \widehat{w}_\lambda(x) = \left(\frac{1+\eta}{1+\lambda}\right)^2w_\lambda\left(\left(\frac{1+\lambda}{1+\eta}\right)x\right), \qquad x\in \overline{\Omega}_\lambda.
\end{equation*}
Using that with $\lambda = \alpha,\beta$ we obtain
\begin{equation*}
    \begin{cases}
    H\left(\left(\frac{1+\alpha}{1+\eta}\right)x, \left(\frac{1+\eta}{1+\alpha}\right) D\widehat{w}_\alpha(x)\right) - \varepsilon \Delta \widehat{w}_\alpha(x)
    \leq c(\alpha) &\qquad\text{in}\; (1+\eta)\Omega,\vspace{0.2cm}\\ 
    H\left(\left(\frac{1+\beta}{1+\eta}\right)x, \left(\frac{1+\eta}{1+\alpha}\right)
    D\widehat{w}_\beta(x)\right) - \varepsilon \Delta \widehat{w}_\beta(x) \leq c(\beta) &\qquad\text{in}\; (1+\eta)\Omega.
    \end{cases}
\end{equation*}
Let $\tilde{\alpha} = \frac{1+\alpha}{1+\eta}$ and $\tilde{\beta} = \frac{1+\beta}{1+\eta}$ then $s\tilde{\alpha}+(1-s)\tilde{\beta}=1$. We can write the equations as 
\begin{equation*}
    \begin{cases}
    H\left(\tilde{\alpha} x, \frac{1}{\tilde{\alpha}} D\widehat{w}_\alpha(x)\right) - \varepsilon \Delta \widehat{w}_\alpha(x)
    \leq c(\alpha) &\qquad\text{in}\; (1+\eta)\Omega,\vspace{0.2cm}\\ 
    H\left(\tilde{\beta} x, \frac{1}{\tilde{\beta}}
    D\widehat{w}_\beta(x)\right) - \varepsilon \Delta \widehat{w}_\beta(x) \leq c(\beta) &\qquad\text{in}\; (1+\eta)\Omega.
    \end{cases}
\end{equation*}
Let us define $\widehat{w} = s\widehat{w}_\alpha + (1-s)\widehat{w}_\beta \in \mathrm{C}(\overline{\Omega}_\eta)$. Using $H(x,\xi) = |\xi|^p - f(x)$, we compute heuristically, assuming $\widehat{w}_{\alpha}$ and $\widehat{w}_\beta$ are differentiable
\begin{equation*}
\begin{aligned}
    &sc(\alpha) + (1-s)c(\beta) 
    \geq sH\left(\tilde{\alpha} x, \frac{1}{\tilde{\alpha}} D\widehat{w}_\alpha(x)\right)  + (1-s)H\left(\tilde{\beta} x, \frac{1}{\tilde{\beta}}
    D\widehat{w}_\beta(x)\right) - \varepsilon \Delta \widehat{w}(x) \\
     &\qquad\qquad = s\left|\frac{D\widehat{w}_\alpha(x)}{\tilde{\alpha}}\right|^p +(1-s) \left|\frac{D\widehat{w}_\beta(x)}{\tilde{\beta}}\right|^p - sf(\tilde{\alpha}x ) - (1-s)f(\tilde{\beta}x) - \varepsilon \Delta\widehat{w}(x)\\
     & \qquad\qquad \geq s\left|\frac{D\widehat{w}_\alpha(x)}{\tilde{\alpha}}\right|^p +(1-s) \left|\frac{D\widehat{w}_\beta(x)}{\tilde{\beta}}\right|^p - f(x) - s(1-s)\left|\frac{\alpha-\beta}{1+\eta}\right||x|^2 - \varepsilon \Delta \widehat{w}(x)
\end{aligned}
\end{equation*}
where we use the semi-concavity of $f(\cdot)$ and $s\tilde{\alpha} + (1-s)\tilde{\beta} = 1$ in the last inequality. We note that the rigorous argument using viscosity solution can be done, for example, following the strategy of \cite[Lemma 2.7]{ishii_vanishing_2017-1}. We use H\"older inequality to get, for $p>2$ and $a,b,x,y > 0$ and $s\in (0,1)$ that
\begin{align*}
     \left(s\frac{a^p}{x^p} + (1-s)\frac{b^p}{y^p}\right)\big(sx^2+(1-s)y^2\big) \big(sx+(1-s)y\big)^{p-2} \geq (sa+(1-s)b)^p.
\end{align*}
Applying this inequality, we obtain
\begin{equation*}
\begin{aligned}
    \left(s\left|\frac{D\widehat{w}_\alpha(x)}{\tilde{\alpha}}\right|^p +(1-s) \left|\frac{D\widehat{w}_\beta(x)}{\tilde{\beta}}\right|^p \right)  \geq \frac{|D\widehat{w}(x)|^p}{s\tilde{\alpha}^2 + (1-s)\tilde{\beta}^2} = \left|\frac{1}{\theta_s} D\widehat{w}(x)\right|^p
\end{aligned} 
\end{equation*}
where 
\begin{equation}\label{eq:theta_s}
    \theta_s = \left(s\tilde{\alpha}^2 + (1-s)\tilde{\beta}^2\right)^{1/p} = \left(1+s(1-s)\frac{(\beta-\alpha)^2}{(1+\eta)^2}\right)^{1/p} \in (1,+\infty).
\end{equation}
We deduce that
\begin{equation*}
    \left|\frac{1}{\theta_s} D\widehat{w}(x)\right|^p - f(x)-\varepsilon\Delta \widehat{w}(x)\leq sc(\alpha) + (1-s)c(\beta) + s(1-s)\left|\frac{\alpha-\beta}{1+\eta}\right|^2|x|^2 \qquad\text{in}\;\Omega_\eta
\end{equation*}
or in other words, we have
\begin{equation*}
    H\left(x, \frac{1}{\theta_s} D\widehat{w}(x)\right) - \varepsilon\Delta \widehat{w}(x) \leq sc(\alpha) + (1-s)c(\beta) + s(1-s)\left|\frac{\alpha-\beta}{1+\eta}\right|^2|x|^2 \qquad\text{in}\;\Omega_\eta.
\end{equation*}
Therefore $H_\phi(x,D\widehat{w}(x)) - \varepsilon\Delta \widehat{w}(x) \leq 0$ in $\Omega_\eta$ where 
\begin{equation*}
    \phi(x,v) = L(x,\theta_s v) + sc(\alpha)+(1-s)c(\beta)  + s(1-s)\left|\frac{\alpha-\beta}{1+\eta}\right|^2|x|^2 \in \Phi^+(\overline{\Omega}_\eta\times\R^n)    
\end{equation*}
thanks to the separable form $H(x,\xi) = |\xi|^p - f(x)$. We, therefore, deduce that
\begin{align*}
    \langle \mu, L(x, \theta_s v)\rangle_{\Omega_\eta} + sc(\alpha)+(1-s)c(\beta)  + s(1-s)\left|\alpha-\beta\right|^2\left\langle \mu, \frac{|x|^2}{1+\eta}\right\rangle_{\Omega_\eta}\geq 0 
\end{align*}
for all $\mu \in \mathcal{M}(\Omega_\eta)$, where $\mathcal{M}(\Omega_\eta)$ is the set of minimizing Mather measures on $\overline{\Omega}_\eta\times\R^n$ as defined in Definition \ref{def:M_0}. We observe that 
\begin{equation*}
    \left\langle \mu, \frac{|x|^2}{1+\eta}\right\rangle_{\Omega_\eta} = \int_{\overline{\Omega}_\eta\times\R^n} \left|\frac{x}{1+\eta}\right|^2\;d\mu (x,v) \leq \widehat{C}_1(\overline{\Omega})
\end{equation*}
where $\widehat{C}_1(\overline{\Omega})$ depends only on the size of $\Omega$. Using $\langle \mu, L(x,  v)\rangle_{\Omega_\eta} = -c(\eta)$ we have
\begin{equation*}
\begin{aligned}
    \langle \mu, L(x, \theta_s v) - L(x,v)\rangle_{\Omega_\eta} &+ sc(\alpha)+(1-s)c(\beta) + C_1(\overline{\Omega}) s(1-s)|\alpha-\beta|^2\geq c(\eta) 
\end{aligned}
\end{equation*}
for all $\mu \in \mathcal{M}(\Omega_\eta)$. Using \eqref{eq:theta_s}, we estimate
\begin{equation*}
    L(x,\theta_s v) - L(x,v) = |\theta_s v|^q  - |v|^q = |v|^q (|\theta_s|^q - 1) = |v|^q \left[\left(1+ s(1-s)\frac{|\beta-\alpha|^2}{(1+\eta)^2}\right)^\frac{q}{p} - 1\right]
\end{equation*}
Therefore
\begin{equation*}
    \langle \mu, L(x,\theta_s v) - L(x,v) \rangle_{\Omega_\eta} = \langle \mu, |v|^q\rangle_{\Omega_\eta} \left[\left(1+ s(1-s)\frac{|\beta-\alpha|^2}{(1+\eta)^2}\right)^\frac{q}{p} - 1\right].
\end{equation*}
Using Bernoulli's inequality $(1+x)^r \leq 1+rx$ if $r\in (0,1)$ and $x>-1$, we deduce that
\begin{equation*}
    \left[\left(1+ s(1-s)\frac{|\beta-\alpha|^2}{(1+\eta)^2}\right)^\frac{q}{p} - 1\right]  \leq \left(\frac{q}{p}\right)s(1-s)\frac{|\beta-\alpha|^2}{(1+\eta)^2}
\end{equation*}
since $p/q < 1$ as $q\leq 2 < p$. 
\medskip

Using the fact that $\langle \mu, L\rangle_{\Omega_\eta} = -c(\eta)$ for $\mu\in \mathcal{M}(\Omega_\eta)$ and $L(x,v) = C_p|v|^q + f(x)$ we deduce that $\langle \mu, C_q|v|^q \rangle_{\Omega_\eta} + \langle \mu, f\rangle_{\Omega_\eta} = - c(\eta)$ and thus
\begin{equation*}
    \left\langle \mu, |v|^q\right\rangle_{\Omega_\eta}  = -\frac{c(\eta) + \langle \mu, f\rangle_{\Omega_\eta}}{C_q} \leq \widehat{C}_2(\overline{U}, f, q):= \frac{1}{C_q}\left(\max_{\overline{U}} |f| + C_M(\overline{U})\right)
\end{equation*}
where $C_M(\overline{U})$ is the constant defined in Corollary \ref{cor:bound_delta-u_delta} and $U$ is the bigger domain where all variations $\Omega_\eta \subset U$. Therefore we obtain 
\begin{equation*}
    c(\eta) \leq sc(\alpha) + (1-s)c(\beta) + \left[\widehat{C}_2(\overline{U}, f, q) \left(\frac{q}{p}\right) + \widehat{C}_1(\overline{U}) \right]s(1-s)|\beta-\alpha|^2 
\end{equation*}
for all $\alpha<\beta$ and $\eta \in (\alpha, \beta)$ with $\eta = s\alpha + (1-s)\beta$ for $s\in [0,1]$.
\end{proof}

\subsection{The quadratic case}\label{sec:the quad case}
% The differentiability of the map $\lambda\mapsto c(\lambda)$ is however much simpler when $p=2$ thanks to the connection between the quadratic Hamilton--Jacobi equation and linear elliptic equation via the Hopf-Cole transformation. For the existence of an eigenvalue and related wellposedness in the case $p=2$, we refer the readers to \cite{lasry_nonlinear_1989} or Appendix of \cite{YuTu2021large}.

The differentiability of $\lambda\mapsto c(\lambda)$ is simpler when $p=2$, owing to the connection between the quadratic Hamilton-Jacobi equation and linear elliptic equation via the Hopf-Cole transformation. For the existence of an eigenvalue and related well-posedness in the $p=2$ case, see \cite{lasry_nonlinear_1989} or the Appendix of \cite{YuTu2021large}. Here, we present a formal argument for the existence and representation of $c'(0)$. A comprehensive examination of $1<p\leq 2$ is planned for future investigation.

We recall that $r(\lambda) = \lambda$ and $\Omega_\lambda = (1+\lambda)\Omega$. If $p=2$, the equation that defines the eigenvalue on $\Omega_\lambda$ is
\begin{equation}\label{eq:EPp=2}
    \begin{cases}
    \begin{aligned}
        |Dv(x)|^2 - f(x) - \varepsilon \Delta v(x) = c(\lambda) 
            &\qquad\text{in}\;\Omega_\lambda,\\
        v(x) = +\infty  &\qquad\text{on}\;\partial\Omega_\lambda.
    \end{aligned}
    \end{cases}
\end{equation}
% Solutions to \eqref{eq:EPp=2} are unique up to adding a constant and they are bounded from below (see \cite{lasry_nonlinear_1989}). Let us choose the solution $\hat{v}_\lambda$ of \eqref{eq:EPp=2} such that
Solutions to \eqref{eq:EPp=2} are unique up to adding a constant and bounded from below (see \cite{lasry_nonlinear_1989}). Let us choose the solution $\hat{v}_\lambda$ of \eqref{eq:EPp=2} such that
\begin{equation}\label{eq:normalized_v_lambda}
    \int_{\Omega_\lambda} \left|e^{-\hat{v}_\lambda(x)/\varepsilon}\right|^2 dx = 1.
\end{equation}
Using the Hopf-Cole transform, we define $w:\overline{\Omega}\to \R$ as $w(x) = e^{-\hat{v}_\lambda(x)/\varepsilon}$ for $x\in \Omega_\lambda$, leading to a linear problem
\begin{equation}\label{eq:v^eps_a}
 \begin{cases}
 \begin{aligned}
     -\varepsilon^2 \Delta w(x) + f(x)w(x) &= c(\lambda) w(x) &\qquad &\text{in}\;\Omega_\lambda,\\ \vspace{0.2cm}
     w(x) &= 0 &\qquad &\text{on}\;\partial\Omega_\lambda.
 \end{aligned}
 \end{cases}
\end{equation}
Therefore, $c(\lambda)$ is an eigenvalue of the linear operator $\mathcal{L}[w] = (-\varepsilon^2 \Delta + f)w$ on $\Omega_\lambda$ with Dirichlet boundary condition. Furthermore, by using \eqref{eq:c(0)}, it is clear that $c(\lambda)$ is the principal eigenvalue of the linear problem \eqref{eq:v^eps_a}. Thus, it is a simple eigenvalue, and it admits a variational presentation
\begin{equation}\label{eq:new_c(lambda)}
    c(\lambda) = \min \left\lbrace \int_{\Omega_\lambda} \left(\varepsilon^2 |Du(x)|^2  + f(x)|u(x)|^2\right)dx: u\in H^1_0(\Omega_\lambda), \Vert u\Vert_{L^2(\Omega_\lambda)}=1 \right\rbrace.
\end{equation}
Let $w_\lambda$ be the unique solution to \eqref{eq:v^eps_a} with $\Vert w_\lambda\Vert_{L^2(\Omega_\lambda)} = 1$. \medskip

Using tools from \emph{shape analysis}, for instance, see \cite{delfour_shapes_2011,  henrot_extremum_2006, henrot_variation_2005, sz} and the references therein, we can obtain some information about the regularity of $\lambda\mapsto c(\lambda)$ and $\lambda\mapsto w_\lambda(\cdot)$. This problem is often examined in the literature within a slightly broader context. We assume that each point $x$ in the domain $\overline{\Omega}$ is changed under a smooth, autonomous (time-independent) velocity field $\mathbf{V}(x) \in \mathrm{C}^k(\overline{\Omega};\R^n)$. Consider the following transformation, which is close to a perturbation of the identity (see \cite{sz}), defined by $T_{\lambda}(x) = x + \lambda\mathbf{V}(x)$ for $x\in \overline{\Omega}$ where $\lambda\in I_0\subset\R$ is a neighborhood of $0$. We denote $\Omega_\lambda = T_\lambda(\Omega)$ (our setting is a special case with $\mathbf{V}(x) = x$ for $x\in \R^n$).\medskip 

For simplicity, we henceforth assume that $f = 0$. Using \cite[Theorem 2.5.1]{henrot_extremum_2006} we have $c'(0)$ exists and also $\lambda\mapsto w_\lambda(\cdot))$ is differentiable, as well as $w'(x) = \frac{d}{d\lambda}\left( w_\lambda(x)\right)\big|_{\lambda=0}$ exists and is a function in $H^1(\Omega)$. We will derive a formula for $c'(0)$ using solution $w_0$ to \eqref{eq:v^eps_a}. \medskip

Denote by $\mathbf{n}(x)$ the unit outward normal vector at $x\in \partial\Omega$. We use differentiation with respect to $\lambda$ in \cite[Section 2.31]{sz}. We differentiate $\Vert w_\lambda\Vert_{L^2(\Omega_\lambda)} = 1$ with respect to $\lambda$ to obtain that 
\begin{align}\label{eq:ww'=0}
    0 = \int_\Omega 2w_0(x)w'(x)\;dx + \int_{\partial\Omega} |w_0(x)|^2 \mathbf{V}(x)\cdot \mathbf{n}(x)\;dx = \int_\Omega 2w_0(x)w'(x)\;dx,
\end{align}
since $w_0 = 0$ on $\partial\Omega$. Next, we differentiable $c(\lambda)$ in \eqref{eq:new_c(lambda)} with $f = 0$ to obtain that 
\begin{align}
    c'(0) 
    &= 2\varepsilon^2\int_\Omega  Dw_0(x)\cdot Dw'(x)\;dx  +\varepsilon^2 \int_{\partial\Omega}|Dw_0(x)^2| \mathbf{V}(x)\cdot \mathbf{n}(x)\;dS(x)\nonumber\\
    &=  2\varepsilon^2\int_{\partial\Omega}  w'(x)\frac{\partial w_0}{\partial \textbf{n}}(x)\;dS(x) +\varepsilon^2\int_{\partial\Omega} |Dw_0(x)^2| \mathbf{V}(x)\cdot \mathbf{n}(x)\;dS(x) \label{eq:compute_c(0)-1},
\end{align}
where we use integration by parts for $-\varepsilon^2 \Delta w_0 = c(0)w_0$ and \eqref{eq:ww'=0}. On the other hand, by differentiating $-\varepsilon^2 \Delta w_\lambda = c(\lambda)w_\lambda$, the equation for $w'$ reads
\begin{equation*}
    -\varepsilon \Delta w' = c'(0)w_0 + c(0)w' \qquad\text{in}\;\Omega.
\end{equation*}
Multiply both sides by $w_0$, using integration by parts with \eqref{eq:ww'=0} and $\Vert w_0\Vert_{L^2(\Omega)}=1$ we obtain
\begin{equation}\label{eq:compute_c(0)-2}
    c'(0) = -\varepsilon^2 \int_\Omega w_0(x)\Delta w'(x)\;dx = \varepsilon^2 \int_\Omega w' \frac{\partial w_0}{\partial \textbf{n}}\;dS(x).
\end{equation}
From \eqref{eq:compute_c(0)-1} and \eqref{eq:compute_c(0)-2} we deduce that 
\begin{equation*}
    \int_{\partial\Omega} w' \frac{\partial w_0}{\partial \textbf{n}}\;dS(x) = -\int_{\partial \Omega } |Dw_0|^2 \textbf{V}\cdot \textbf{n}\;dS(x)
\end{equation*}
and thus, we conclude that 
\begin{equation*}
    c'(0) = -\varepsilon^2 \int_{\partial \Omega } |Dw_0|^2 \textbf{V}\cdot \textbf{n}\;dS(x) = -\varepsilon^2 \int_{\partial\Omega} \left|\frac{\partial w_0}{\partial \textbf{n}}(x)\right|^2 \left(x\cdot \textbf{n}\right)\;dS(x).
\end{equation*}

We also have higher derivatives of $\lambda\mapsto c(\lambda)$, as in \cite{henrot_extremum_2006}.
\begin{cor}\label{thm:p=2} If $p=2$, $f\equiv 0$ and $r(\lambda) = \lambda$, then the map $c\mapsto c(\lambda)$ is twice differentiable everywhere.
\end{cor}

%At the end of the appendix it has to be set to the previous value
\addtocontents{toc}{\protect\setcounter{tocdepth}{1}}

\appendix
\addcontentsline{toc}{section}{Appendix}

%Here's the place to hide section from ToC 
\addtocontents{toc}{\protect\setcounter{tocdepth}{0}}

\section{State-constraint solutions for Hamilton--Jacobi equation}
\label{ap:a}

\subsection{Assumptions}
We present this section as a self-contained section. For generality, we state the results under a more general set of assumptions.
\begin{description}[style=multiline, labelwidth=1cm, leftmargin=2.0cm]
    \item[\namedlabel{itm:A21}{$(\mathcal{H}_1)$}] $\xi\mapsto H(x,\xi)$ is convex for every fixed $x\in \overline{\Omega}$.
    \item[\namedlabel{itm:A22}{$(\mathcal{H}_2)$}] There exists $p>2$ such that, for every $R>0$, there are $0< a_R \leq 1\leq b_R$ such that for $x,y\in B_R$ and all $\xi,\xi_1, \xi_2\in \mathbb{R}^n$ there holds
    \begin{align}
        & a_R |\xi|^p - b_R \leq H(x,\xi) \leq \Lambda_1(|\xi|^p + 1)\\
        & |H(x,\xi) - H(y,\xi)| \leq (\Lambda_1|\xi|^p + b_R)|x-y|\\
        & |H(x,\xi_1) - H(x,\xi_2)| \leq \Lambda_1(|\xi_1|+|\xi_2|+1)^{p-1}|\xi_1-\xi_2|.
     \end{align}
    \item[\namedlabel{itm:A23}{$(\mathcal{H}_3)$}] The matrix $a(x):\mathbb{R}^n\to \mathbb{M}^{n} = \mathbb{R}^{n\times n}$ has a Lipschitz square root $\sigma: \mathbb{R}^n\to \mathbb{R}^{k\times n}$ for $k\in \mathbb{N}$ such that $a(x) = \sigma^T(x)\sigma(x)$ where
    \begin{equation}\label{condA3}
        |\sigma(x)|\leq \Lambda_2, \qquad |\sigma(x) - \sigma(y)| \leq \Lambda_2|x-y|, \qquad x,y \in \R^n
    \end{equation}
    for some constant $\Lambda_2 > 0$.
\end{description}

\subsection{Existence and H\"older estimate}

\begin{defn}\label{defn:1} We consider the following equation with $\delta \geq 0$:
\begin{equation}\label{HJ-static}
    \delta u(x) + H(x,Du(x)) - \varepsilon\,\mathrm{Tr}\big(a(x)D^2 u(x)\big) = 0 \qquad\text{in}\;\Omega \tag{HJ}.
\end{equation}
We say that
\begin{itemize}
\item[(i)] $v\in \mathrm{BUC}(\Omega;\mathbb{R})$ is a viscosity subsolution of \eqref{HJ-static} in $\Omega$ if for every $x\in \Omega$ and $\varphi\in \mathrm{C}^2(\Omega)$ such that $v-\varphi$ has a local maximum over $\Omega$ at $x$ then 
\begin{equation*}
    \delta v(x) + H\big(x,D\varphi(x)\big) -\varepsilon\,\mathrm{Tr}\big(a(x)D^2\varphi (x)\big) \leq 0.
\end{equation*}

\item[(ii)] $v\in \mathrm{BUC}(\overline\Omega;\mathbb{R})$ is a viscosity supersolution of \eqref{HJ-static} on $\overline\Omega$ if for every $x\in \overline{\Omega}$ and $\varphi\in \mathrm{C}^2(\overline{\Omega})$ such that $v-\varphi$ has a local minimum over $\overline{\Omega}$ at $x$ then
\begin{equation*}
    \delta v(x) + H\big(x,D\varphi(x)\big) -\varepsilon\,\mathrm{Tr}\big(a(x)D^2\varphi(x)\big) \geq 0.
\end{equation*}
\end{itemize}

If $v$ is a viscosity subsolution to \eqref{HJ-static} in $\Omega$, and is a viscosity supersolution to \eqref{HJ-static} on $\overline{\Omega}$, i.e.,
\begin{equation}\label{state-def}
\begin{cases}
\delta v(x) + H(x,Dv(x)) - \varepsilon\,\mathrm{Tr}\big(a(x)D^2v(x)\big) \leq 0 &\quad\text{in}\; \Omega,\\
\delta v(x) + H(x,Dv(x))-  \varepsilon\,\mathrm{Tr}\big(a(x)D^2v(x)\big) \geq 0 &\quad\text{on}\; \overline{\Omega},
\end{cases}
\end{equation}
then we say that $v$ is a state-constraint viscosity solution of \eqref{HJ-static}. 
\end{defn}

% \begin{defn}\label{defn:semijets} For a real valued function $w(x)$ define for $x\in \mathcal{O}$ where $\mathcal{O}$ is locally compact in $\R^n$, we define the semi super-jet and $w$ at $x$ as
% \begin{align*}
%     J_{\mathcal{O}}^{2,+}w(x)=
%     \left\lbrace 
%     \begin{array}{l}
%         (p,X)\in \R^n\times \mathbb{S}^n: \\
%         \displaystyle\limsup_{y \rightarrow x} \frac{w(y)-w(x)-p\cdot (y-x)- \frac{1}{2}(y-x)\cdot X(y-x)}{|y-x|^2} \leq 0
%     \end{array}
%     \right\rbrace.
% \end{align*}
% Similarly, we define the semi sub-jet of $w$ at $x$ as $ J_{\mathcal{O}}^{2,-}w(x) = -  J_{\mathcal{O}}^{2,+}(-w)(x)$.
% The closures of semi-jets are defined as follows.
% \begin{align}\label{eq:semijets_closure}
%     \overline{J}_{\mathcal{O}}^{2,\pm}w(x) &= \Big\lbrace (p,X), \exists\;x_k\in \mathcal{O},\; (p_k,X_k)\in J^{2,\pm}_{\mathcal{O}}w(x_k) \quad\text{and}\quad  (p_k,X_k) \to (p,X)
%     \Big\rbrace.
% \end{align}
% \end{defn}

% \begin{rem} The first-order semi-jets can be defined in the same manner, but they are usually called super-differential and sub-differential. For a real-valued function $w:\mathcal{O}\to \R$, we define the superdifferential of $w$ at $x$ as
% \begin{align*}
% J^{1,+} w(x) = D^{+}w(x)&=\left \lbrace p \in \R^n : \limsup_{y \rightarrow x} \frac{w(y)-w(x)-p\cdot (y-x)}{|y-x|} \leq 0\right\rbrace
% \end{align*}
% and the subgradient of $w$ at $x$ is defined by $D^{-}w(x) = -D^{+}(-w)(x)$.
% \end{rem}

We refer the readers to \cite{ bardi_optimal_1997,Barles1994, tran_hamilton-jacobi_2021} for the equivalent definition of viscosity solution using super-differential $ J_{\mathcal{O}}^{2,\pm}$ and sub-differential $D^{\pm}$. Existence, wellposedness, and gradient bound on the solution of \eqref{state-def} and a description of state-constraint boundary condition can be found in \cite{capuzzo-dolcetta_hamiltonjacobi_1990, kim_state-constraint_2020, soner_optimal_1986}. The existence of solutions to \eqref{state-def} can be established using Perron's method. We omit the proof of the following Theorem and refer the readers to \cite[Theorem 4.2]{armstrong_viscosity_2015}.

\begin{thm}\label{thm:Perron} Assume \ref{itm:A21}, \ref{itm:A22}, \ref{itm:A23}. Define 
\begin{equation}\label{eq:perron_def}
    u^\delta(x) = \sup \Big\lbrace w(x):w\in \mathrm{USC}(\overline{\Omega})\;\text{is a subsolution of} \;\eqref{state-def}\;\text{in}\;\Omega\Big\rbrace.
\end{equation}
Then $u^\delta\in \mathrm{C}^{0,1}_{\mathrm{loc}}(\Omega)\cap \mathrm{LSC}(\overline{
\Omega})$ and $u^0$ is a solution of \eqref{state-def}.
\end{thm}

For the superquadratic Hamiltonian, any subsolution is at least H\"older continuous.

\begin{thm}\label{thm:Holder} Under assumptions \ref{itm:A21}, \ref{itm:A22}, \ref{itm:A23}, any upper semicontinuous subsolution $u$ of \eqref{HJ-static} is uniformly H\"older continuous up to the boundary, i.e., $u\in \mathrm{C}^{0,\alpha}(\overline{\Omega})$ with
\begin{equation*}
    |u(x) - u(y)| \leq C_0|x-y|^\alpha, \qquad \alpha = \frac{p-2}{p-1} \in (0,1)
\end{equation*}
where $C_0$ is a constant depending on the diameter of $\Omega,p$ and constants from \ref{itm:A22}, \ref{itm:A23}.
\end{thm}
We refer the reader to \cite{armstrong_viscosity_2015} or \cite{BARLES201031} for the proof of this Theorem. We also refer the readers to \cite{lasry_nonlinear_1989} for a proof in the case $a(x)\equiv 1$ and $H(x,\xi) = |\xi|^p - f(x)$.

\subsection{Comparison principle for H\"older continuous solutions}

We further assume that 
\begin{description}[style=multiline, labelwidth=1cm, leftmargin=2.0cm]
    \item[\namedlabel{itm:A24}{$(\mathcal{H}_4)$}] $H(x,\xi) = \mathcal{H}(\xi) - f(x)$ where $f\in \mathrm{C}^1(\overline{\Omega})$ and $\mathcal{H}$ is homogeneous of degree $p>2$.
\end{description}

The uniqueness of the state-constraint solution to \eqref{HJ-static} follows from a comparison principle. It was first studied in \cite{soner_optimal_1986-1} for the first-order equation under an interior sphere assumption (see also \cite{BARLES201031} for a general nonlinear case with different assumptions). For our purpose of studying the scaling domains, we provide a simple proof using \ref{itm:A1} here for the readers' convenience. We note that in \ref{itm:A1}, the assumption $\Omega$ is star-shaped can be removed, that is, any bounded, open subset of $\mathbb{R}^n$ containing the origin that satisfies \eqref{eq:geometric-condition} for some $\kappa > 0$ is star-shaped (see \cite{tu_vanishing_2021}). 

\begin{thm}[Comparison principle]\label{CP continuous} 
Assume \eqref{eq:geometric-condition}, \ref{itm:A21}, \ref{itm:A22}, \ref{itm:A23}, \ref{itm:A24} and $\lambda>0$. Let $u,v\in \mathrm{C}(\overline{\Omega})$ be a viscosity subsolution in $\Omega$ and supersolution on $\overline{\Omega}$ of \eqref{HJ-static}, respectively. Assume that $u,v$ are H\"older continuous, i.e., belong to $\mathrm{C}^{0,\alpha}(\overline{\Omega})$ for some $\alpha \in (0,1]$, then $u(x)\leq v(x)$ for all $x\in \overline{\Omega}$.
\end{thm}

\begin{proof}
For $\delta>0$ we define $C_1 = (2/\kappa)^{2-\alpha}C_0$ where $\kappa$ is defined as in \ref{itm:A1} and $C_0$ is the H\"older constant of $u$. Let us define
\begin{equation*}
    \Phi(\hat{x},y) = u\left(\frac{\hat{x}}{1+\theta}\right) - v(y) - \frac{C_1|\hat{x}-y|^2}{\theta^{2-\alpha}}, \qquad (\hat{x},y)\in (1+\theta)\overline{\Omega}\times \overline{\Omega}.
\end{equation*}
Assume $\Phi$ achieves its maximum over $(1+\theta)\overline{\Omega}\times \overline{\Omega}$ at $(\hat{x}_\theta,y_\theta)\in(1+\theta)\overline{\Omega}\times \overline{\Omega}$. We claim that $\hat{x}_\theta \in (1+\theta)\Omega$. From $\Phi(\hat{x}_\theta,y_\theta) \geq \Phi(y_\theta,y_\theta)$ we obtain
\begin{equation*}
     \frac{C_1|\hat{x}_\theta - y_\theta|^2}{\theta^{2-\alpha}} \leq  u\left(\frac{\hat{x}_\theta}{1+\theta}\right)  -  u\left(\frac{y_\theta}{1+\theta}\right)  \leq \frac{C_0|\hat{x}_\theta-y_\theta|^\alpha}{(1+\theta)^\alpha} \leq C_0|\hat{x}_\theta-y_\theta|^\alpha.
\end{equation*}
As a consequence
\begin{equation}\label{eq:distxydelta}
     \mathrm{dist}(\hat{x}_\theta,\overline{\Omega})\leq  |\hat{x}_\theta-y_\theta| \leq \left(\frac{C_0}{C_1}\right)^{\frac{1}{2-\alpha}}\theta = \left(\frac{\kappa}{2}\right)\theta.
\end{equation}
From \ref{itm:A1} we conclude $\hat{x}_\theta \in (1+\theta)\Omega$. This yields that
\begin{equation*}
    (\hat{x},y)\mapsto u\left(\frac{\hat{x}}{1+\theta}\right) - v(y) -  \frac{C_1|\hat{x}-y|^2}{\theta^{2-\alpha}},\qquad\text{has a maximum over}\;(1+\theta)\overline{\Omega}\times\overline{\Omega}.
\end{equation*}
at $(\hat{x}_\theta,y_\theta)\in (1+\theta)\Omega\times\overline{\Omega}$. Let $\hat{x} = (1+\theta) x$ for $x\in \Omega$ then $\hat{x}\mapsto x$ is a bijection from $(1+\theta)\overline{\Omega}$ to $\overline{\Omega}$. Equivalently, we deduce that
\begin{equation*}
    (x,y)\mapsto u\left(x\right) - v(y) -  \frac{C_1|(1+\theta)x-y|^2}{\theta^{2-\alpha}}\qquad\text{has a maximum over}\;\overline{\Omega}\times\overline{\Omega}
\end{equation*}
at $(x_\theta,y_\theta)\in \Omega\times\overline{\Omega}$. Let
\begin{equation*}
    \phi(x,y) = \frac{C_1|(1+\theta)x-y|^2}{\theta^{2-\alpha}}, \qquad (x,y)\in \overline{\Omega}\times\overline{\Omega}.
\end{equation*}
From Lemma \ref{lem:ishii} below, we can find $X_\theta,Y_\theta\in \mathbb{S}^n$ such that
\begin{equation*}
     \Big(u(x_\theta),D_x\phi(x_\theta,y_\theta),X_\theta\Big) \in \overline{J}^{2,+}_{\overline{\Omega}}u(x_\theta) , \qquad \Big(v(y_\theta),-D_y\phi(x_\theta,y_\theta),Y_\theta\Big) \in \overline{J}^{2,-}_{\overline{\Omega}}v(y_\theta)
\end{equation*}
and 
\begin{equation*}
    \xi^T(X_\theta - Y_\theta)\xi \leq 2C_1\theta^\alpha |\xi|^2 \qquad\text{for all}\;\xi\in\R^n.
\end{equation*}
Choose $\xi = \sigma_i(x)$ in \eqref{eq:lemIshi_b} where $\sigma_i(x)$ is the $i^{\text{th}}$-row vector of $\sigma(x)$ for $i=1,2,\ldots, k$ and taking the sum we deduce that
\begin{equation*}
    \mathrm{Tr}(a(x)(X_\theta-Y_\theta)) = \frac{1}{2}\mathrm{Tr}(\sigma^T(x)(X_\theta-Y_\theta)\sigma(x)) = \sum_{i=1}^k \sigma_i^T(x)(X_\theta - Y_\theta)\sigma_i(x) \leq k \Lambda_2 C_1\theta^\alpha.
\end{equation*}
Here, $\Lambda_2$ is given in \eqref{condA3}. The subsolution test of $u$ at $x_\theta$ and the supersolution test of $v$ at $y_\theta$ give
\begin{align}
    &\delta u(x_\theta) + (1+\theta)^p \mathcal{H}\left(\frac{2C_1((1+\theta)x_\theta - y_\theta)}{\theta^{2-\alpha}}\right) -f(x_\theta) - \varepsilon\, \mathrm{Tr}(a(x)X_\theta) \leq 0, \label{eq:sub1}\\
    &\delta v(y_\theta) + \mathcal{H}\left(\frac{2C_1((1+\theta)x_\theta - y_\theta)}{\theta^{2-\alpha}}\right) -f(y_\theta) - \varepsilon\, \mathrm{Tr}(a(x)Y_\theta)  \geq 0. \label{eq:super1}
\end{align}
Subtract \eqref{eq:super1} from \eqref{eq:sub1} and use the fact that $1\leq (1+\theta)^p$ we deduce that
\begin{align}
    \delta u(x_\theta) - \delta v(y_\theta) &\leq  f(x_\theta) - f(y_\theta) +\varepsilon \;\mathrm{Tr}(a(x)(X_\theta-Y_\theta)) \\
    & + \Big(1-(1+\theta)^p\Big) \mathcal{H} \left(\frac{2C_1((1+\theta)x_\theta - y_\theta)}{\theta^{2-\alpha}}\right) \nonumber\\
    &\leq \omega_f(|x_\theta - y_\theta|) +\varepsilon\Big(k \Lambda_2 C_1 \theta^\alpha\Big) \label{eq:mas}
\end{align}
where $\omega_f(\cdot)$ is the modulus of continuity of $f$. By compactness of $\overline{\Omega}$ we can assume $(x_\theta,y_\theta)\to (z,z)$ as $\theta\to 0$. Assume $x_0\in \overline{\Omega}$ such that $\max_{x\in \overline{\Omega}}(u-v) = u(x_0) - v(x_0)$. Using $\Phi(\hat{x}_\theta,y_\theta)\geq \Phi(\hat{x}_0,x_0)$ where $\hat{x}_0 = (1+\theta)x_0$ we have
\begin{equation*}
    u(x_\theta) - v(y_\theta) \geq u(x_0) - v(y_0) - C_1\theta^\alpha|x_0|^2
\end{equation*}
Let $\theta\to 0$ we deduce that $u(x_0) - v(y_0)  \geq u(z) - v(z)  \geq u(x_0) - v(y_0)$. Therefore we obtain $(u-v)(z) = \max_{\overline{\Omega}}(u-v)$, i.e.,
\begin{equation*}
    \lim_{\theta \to 0} \big(u(x_\theta) - v(y_\theta)\big) = \max_{\overline{\Omega}} (u-v).
\end{equation*}
Using this in \eqref{eq:mas} we obtain $\max_{\overline{\Omega}} (u-v) \leq 0$ and thus $u\leq v$ on $\overline{\Omega}$.
\end{proof}

\begin{lem}\label{lem:ishii} Under the setting in the proof of Theorem~\ref{CP continuous}, there exists $X_\theta,Y_\theta\in \mathbb{S}^n$ such that 
\begin{equation*}
    \Big(u(x_\theta),+D_x\phi(x_\theta,y_\theta),X_\theta\Big) \in \overline{J}^{2,+}_{\overline{\Omega}}u(x_\theta), \qquad
    \Big(v(y_\theta),-D_y\phi(x_\theta,y_\theta),Y_\theta\Big) \in \overline{J}^{2,-}_{\overline{\Omega}}v(y_\theta),
\end{equation*}
and 
\begin{equation}\label{eq:lemIshi_b}
    \xi^T(X_\theta - Y_\theta)\xi \leq 2C_1\theta^\alpha|\xi|^2, \qquad \text{for all}\;\xi\in \R^n.
\end{equation}
\end{lem}
\begin{proof} By Ishii's Lemma (see \cite{Crandall1992} or \cite[Theorem 3.6]{koike_beginners_2014}), for each $\mu>1$ there are $X_\mu,Y_\mu\in \mathbb{S}^n$ such that
\begin{equation*}
    \big(u(x_\theta),D_x\phi(x_\theta,y_\theta),X_\mu\big) \in \overline{J}^{2,+}_{\overline{\Omega}}u(x_\theta), \qquad \big(v(y_\theta),-D_y\phi(x_\theta,y_\theta),Y_\theta\big) \in \overline{J}^{2,-}_{\overline{\Omega}}v(y_\theta),
\end{equation*}
and
\begin{equation}\label{eq:matrix1}
    -(\mu+\Vert \mathbf{A}\Vert)\left(\begin{array}{cc}
        I_n & 0 \\
        0   & I_n
    \end{array}\right) \preceq \left(\begin{array}{cc}
        X_\mu & 0 \\
        0   & -Y_\mu
    \end{array}\right) \preceq \mathbf{A}\left(I_{2n} + \frac{1}{\mu} \mathbf{A}\right),
\end{equation}
where $\mathbf{A} = D^2\phi(x,y)\in \mathbb{S}^{2n}$. We compute from the definition of $\phi$ that
\begin{align*}
    \mathbf{A} &= \frac{C_1}{\theta^{2-\alpha}}\left(\begin{array}{cc}
        \;\;\;\,(1+\theta)^2I_n & -(1+\theta)I_n \\ 
         -(1+\theta)I_n & \;\;\;\; I_n
    \end{array}\right) \\
    &= \frac{C_1}{\theta^{2-\alpha}} \left[\left(\begin{array}{cc}
        \;\;I_n & -I_n  \\
        -I_n & \;\;\;I_n
    \end{array}\right) + \theta \left(\begin{array}{cc}
        2I_n & -I_n   \\
        -I_n & 0
    \end{array}\right) +\theta^2\left(\begin{array}{cc}
        I_n & 0   \\
        0 & 0
    \end{array}\right) \right].
\end{align*}
The first two terms annihilate $(\xi,\xi)^T$ for $\xi\in \R^N$. Therefore
\begin{equation*}
    (\xi,\xi) \mathbf{A}(\xi,\xi)^T = \frac{C_1}{\theta^{2-\alpha}}\Big[\theta^2|\xi|^2\Big] = C_1\theta^\alpha|\xi|^2.
\end{equation*}
Similarly, we have
\begin{equation*}
\begin{split}
    \mathbf{A}^2 &= \left(\frac{C_1}{\theta^{2-\alpha}}\right)^2\left(\begin{array}{cc}
        (\theta^4+4\theta^3 + 7\theta^2 + 6\theta + 1)I_n & -(\theta^3+3\theta^2+4\theta+1)I_n  \\
         \qquad -(\theta^3+3\theta^2+4\theta+1)I_n &  \qquad\quad\; (\theta^2+2\theta+2)I_n 
    \end{array}\right)\\
    &= \frac{C_1^2}{\theta^{4-2\alpha}}\left[\left(\begin{array}{cc}
        \;\;I_n & -I_n  \\
        -I_n & \;\;\;2I_n
    \end{array}\right) +2\theta \left(\begin{array}{cc}
        3I_n & -2I_n  \\
         -2I_n & I_n
    \end{array}
    \right) + \theta^2 \left(\begin{array}{cc}
        7I_n & -3I_n  \\
         -3I_n & I_n
    \end{array}
    \right) \right. \\
    & \qquad\qquad \qquad\qquad\qquad\qquad+ \left. 
    \theta^3 \left(\begin{array}{cc}
        4I_n & -1I_n  \\
         -1I_n & 0
    \end{array}\right) + \theta^4 \left(\begin{array}{cc}
        I_n & 0 \\
         0 & 0
    \end{array}\right) 
    \right].
\end{split}
\end{equation*}
Using $(\xi,\xi)^T$ for $\xi\in \R^n$ the previous equation we deduce that
\begin{equation*}
    (\xi,\xi) \mathbf{A}^2 (\xi,\xi)^T = \frac{C_1^2}{\theta^{4-2\alpha}}\Big(1+0 +2\theta^2+2\theta^3 + \theta^4\Big)|\xi|^2.
\end{equation*}
We conclude that for $\theta$ small then
\begin{equation*}
    \xi^T(X_\mu-Y_\mu)\xi \leq \left(C_1\theta^\alpha + \frac{2C_1^2}{\mu \theta^{4-2\alpha}}\right)|\xi|^2.
\end{equation*}
Choose $\mu = \mu(\theta)$ such that $\mu = 2C_1\theta^{\alpha-4} > 1$ we obtain the conclusion \eqref{eq:lemIshi_b}. 
\end{proof}

From Theorem \ref{thm:Perron}, the local H\"older continuity of solutions, and the comparison principle, we obtain the following corollary.
\begin{cor}\label{cor:bound_delta-u_delta} Assume \eqref{eq:geometric-condition}, \ref{itm:A21}, \ref{itm:A22}, \ref{itm:A23}, \ref{itm:A24} and $\lambda>0$. There exists a unique solution $u_\delta\in \mathrm{C}^{0,\alpha}(\overline{\Omega})$ of \eqref{HJ-static}.
\end{cor}
\begin{proof} Let $u_\delta\in \mathrm{C}(\overline{\Omega})$ be the viscosity solution defined as in \eqref{eq:perron_def}. By Theorem \ref{thm:Holder} we have $u^\delta\in \mathrm{C}^{0,\alpha}(\overline{\Omega})$ where $\alpha = (p-2)/(p-1)$. If $v\in \mathrm{C}(\overline{\Omega})$ is another viscosity solution to \eqref{HJ-static} then by Theorem \ref{thm:Holder} $v\in \mathrm{C}^{0,\alpha}(\overline{\Omega})$ as well, thus by the comparison principle in Theorem \ref{CP continuous} we obtain $u_\delta\equiv v$.
\end{proof}

\begin{prop}\label{pro:bound_on_sln}
Assume \eqref{eq:geometric-condition}, \ref{itm:A21}, \ref{itm:A22}, \ref{itm:A23}, \ref{itm:A24} and $\lambda>0$.
\begin{itemize}
    \item[(i)] We have the uniform bound from below $\delta u_\delta(\cdot) \geq -\max_{x\in \overline{\Omega}}H(x,0)$ for all $\delta>0$.
    \item[(ii)] If $B(0,\kappa) \subset \Omega$ then $\delta u_\delta(x) \leq C(\kappa)$ for $x\in \overline{B(0,\kappa)}$, where $C$ depends on $\kappa$ and constants from \ref{itm:A21},\ref{itm:A22}, \ref{itm:A23}.
\end{itemize}
\end{prop}
\begin{proof} Since $u\equiv -\delta^{-1}\max \{H(x,0): x\in \overline{\Omega}\}$ is a classical subsolution in $\Omega$, the comparison principle concludes the lower bound. 

Let $v\in \mathrm{C}(\overline{B(0,\kappa)})$ be a solution to the state-constraint solution to the problem \eqref{HJ-static} with discount factor $\delta=1$ and on the domain $B(0,\kappa)$, i.e.,
\begin{equation*}
\begin{cases}
    v(x) + H\big(x,Dv(x)\big) -\varepsilon\,\mathrm{Tr}\big(a(x)D^2v (x)\big) \leq 0\qquad\text{in}\;B(0,\kappa),\\
    v(x) + H\big(x,Dv(x)\big) -\varepsilon\,\mathrm{Tr}\big(a(x)D^2v (x)\big) \geq 0\qquad\text{in}\;\overline{B(0,\kappa)}.
\end{cases}
\end{equation*}
We observe that 
\begin{equation*}
    w(x) = v(x) + (1+\delta^{-1})\Vert v\Vert _{L^\infty(\overline{B(0,\kappa)})}  , \qquad x\in \overline{B(0,\kappa)}
\end{equation*}
is a supersolution to \eqref{state-def}, thus by comparison principle on $B(0,\kappa)$ we have the conclusion.
\end{proof}

\section{Proofs of some technical results}
\label{ap:c}

\begin{proof}[Proof of Theorem \ref{thm:Holder_est_nested_domain}] Since $u$ is a subsolution to \eqref{eq:reference_problem}, by comparison principle we have $u\leq v$ on $\overline{\Omega}$. Let $\tilde{u}(x) = (1+\theta)^{-2}u((1+\theta)x)$ for $x\in \overline{\Omega}$. We define 
\begin{equation*}
    \Psi(\hat{x},y) = v\left(\frac{\hat{x}}{1+\theta}\right) - \tilde{u}(y) - \frac{C_1|\hat{x}-y|^2}{\theta^{2-\alpha}}, \qquad (\hat{x},y) \in (1+\theta)\overline{\Omega}\times \overline{\Omega}
\end{equation*}
where $C_1 = (2/\kappa)^{2-\alpha}C_0$. Assume $\Psi$ has a maximum at $(\hat{x}_\theta, y_\theta) \in (1+\theta)\overline{\Omega}\times \overline{\Omega}$. Using $\Psi(\hat{x}_\theta, y_\theta)  \geq (y_\theta, y_\theta) $ we deduce that
\begin{equation*}
    \frac{C_1|\hat{x}_\theta-y_\theta|^2}{\theta^{2-\alpha}} \leq v\left(\frac{\hat{x}_\theta}{1+\theta}\right) - v\left(\frac{\hat{y}_\theta}{1+\theta}\right) \leq C_0|\hat{x}_\theta-y_\theta|^{2-\alpha}
\end{equation*}
where $v$ is H\"older continuous (Theorem \ref{thm:Holder_main}). Therefore $\mathrm{dist}(\hat{x}_\theta, \overline{\Omega}) \leq|\hat{x}_\theta - y_\theta|\leq  \frac{\kappa}{2}\theta$. By \ref{itm:A1} we obtain $\hat{x}_\theta \in (1+\theta)\Omega$. Let $\hat{x} = (1+\theta)x$ for $x\in \Omega$, we conclude that
\begin{equation*}
    (x,y)\mapsto v(x) - \tilde{u}(y) - \frac{C_1|(1+\theta)x-y|^2}{\theta^{2-\alpha}} 
\end{equation*}
has a maximum over $\overline{\Omega}\times \overline{\Omega}$ at $(x_\theta,y_\theta)\in \Omega\times \overline{\Omega}$. From Ishii's lemma there exist matrices $X_\theta,Y_\theta \in \mathbb{S}^n$ such that $\xi^T (X_\theta -Y_\theta)\xi \leq 2C_1\theta^\alpha |\xi|^2$ for all $\xi\in \R^n$ and 
\begin{align*}
    &\delta v(x_\theta) + (1+\theta)^p \left|\frac{2C_1((1+\theta)x_\theta - y_\theta)}{\theta^{2-\alpha}}\right|^p - f(x_\theta) - \varepsilon\;\mathrm{Tr}(X_\theta) \leq 0,\\ 
    &\delta(1+\theta)^2 \tilde{u}(y_\theta) + (1+\theta)^p \left|\frac{2C_1((1+\theta)x_\theta - y_\theta)}{\theta^{2-\alpha}}\right|^p - f((1+\theta)y_\theta) - \varepsilon\;\mathrm{Tr}(Y_\theta) \geq 0.
\end{align*}    
Therefore
\begin{equation*}
\begin{aligned}
    \delta v(x_\theta) - \delta u((1+\theta)y_\theta) \leq f(x_\theta) - f((1+\theta)y_\theta) + \varepsilon\;\mathrm{Tr}(X_\theta - Y_\theta)\leq C\theta + \varepsilon C \theta^\alpha
\end{aligned}
\end{equation*}
where $C$ depends on $\mathrm{diam}(\Omega)$, the Lipschitz constant of $f$ on $U$ and $p$. Finally, for $x\in \overline{\Omega}$ then $\Psi(\hat{x}_\theta, y_\theta) \geq \Psi((1+\theta)x, x)$ implies that 
\begin{equation*}
     v(x) - \frac{1}{(1+\theta)^2}u((1+\theta)x) \leq v(x_\theta) -  \frac{1}{(1+\theta)^2}u((1+\theta)x_\theta) + C_1(1+\theta)^2|x|^2\theta^\alpha.
\end{equation*}
Therefore
\begin{equation*}
\begin{aligned}
    \delta v(x) - \delta u(x) - C|\theta|^\alpha  &\leq (1+\theta)^2 \delta v(x) - \delta u((1+\theta)x) \\
    % &\leq (1+\theta)^2 \delta v(x_\theta) - \delta u((1+\theta)x_\theta) + 4C_1\delta \theta^\alpha |x|^2\\
    &\leq  \big(\delta v(x_\theta) -\delta u((1+\theta)x_\theta)\big) + \Big((1+\theta)^2-1\Big)\delta v(x_\theta) +  4C_1\delta \theta^\alpha |x|^2\\
    % & \leq C\theta + C\theta^\alpha + 2\theta|\delta v(x_\theta)| +  4C_1\delta \theta^\alpha |x|^2\\
    &\leq C\theta + C\theta^\alpha + 2C\theta + 4C_1\delta \theta^\alpha|x|^2
\end{aligned}
\end{equation*}
where we invoke Corollary \ref{cor:improved_bounds} to obtain $|\delta v(x_\theta)|\leq C$. 
\end{proof}

\addtocontents{toc}{\protect\setcounter{tocdepth}{0}}
\section*{Acknowledgments}

The authors thank Hung V. Tran for his useful comments. Son N.T. Tu thanks Yeoneung Kim for helping with correcting typos. The authors are grateful to the Institute for Pure \& Applied Mathematics (IPAM, Los Angeles) for the support and hospitality during the program 
High Dimensional Hamilton--Jacobi PDEs, where the work on this article was started.

%At the end of the appendix it has to be set to the previous value
\addtocontents{toc}{\protect\setcounter{tocdepth}{1}}

\bibliography{references.bib}{}
\bibliographystyle{siam}

\end{document}